\documentclass[bj]{imsart}

\RequirePackage[OT1]{fontenc}
\RequirePackage{amsthm,amsmath
}

\usepackage{amsmath,amstext,amssymb,amsopn,amsthm}
\usepackage{url,verbatim}
\usepackage{mathtools}
\usepackage{enumerate}

\usepackage{color,graphicx}

\usepackage[margin=30mm]{geometry}
\usepackage{eucal,mathrsfs,dsfont}

\numberwithin{equation}{section}

\allowdisplaybreaks

\usepackage[utf8]{inputenc}
\usepackage{graphicx,amsmath,amssymb,enumerate,color,amsthm}
\usepackage{epstopdf}
\renewcommand{\P}{\mathbb{P}}

\newcommand{\IGFT}{{\rm IED}}
\newcommand{\IED}{{\rm IED}}

\newcommand{\essinf}{\mathrm{ess\,inf}}

\newcommand{\R}{\mathbb{R}}

\newcommand{\Z}{\mathbb{Z}}
\newcommand{\E}{\mathbb{E}}
\newcommand{\1}{\mathbf{1}}

\newcommand{\la}{\leftarrow}

\newtheorem{thm}{Theorem}[section]
\newtheorem{lemma}[thm]{Lemma}
\newtheorem{corol}[thm]{Corollary}
\newtheorem{propo}[thm]{Proposition}

\theoremstyle{definition}

\newtheorem{defin}[thm]{Definition}

\newtheorem{remark}[thm]{Remark}
\newtheorem{example}[thm]{Example}

\newcommand{\eps}{\varepsilon}

\begin{document}

\begin{frontmatter}
\title{Inverse Exponential Decay: Stochastic Fixed Point Equation and ARMA Models}
\runtitle{Inverse Exponential Decay}

\begin{aug}
\author{\fnms{Krzysztof} \snm{Burdzy}\thanksref{a,e1}\ead[label=e1,mark]{burdzy@uw.edu}}
\author{\fnms{Bartosz} \snm{Ko\L{}odziejek}\thanksref{b,e2}\ead[label=e2,mark]{b.kolodziejek@mini.pw.edu.pl}}
\and
\author{\fnms{Tvrtko} \snm{Tadi\'c}\thanksref{c,e3}%
\ead[label=e3,mark]{tvrtko@math.hr}}

\address[a]{Department of Mathematics, Box 354350, University of Washington, Seattle, WA 98195, USA
\printead{e1}}

\address[b]{Faculty of Mathematics and Information Science, Warsaw University of Technology, Koszykowa 75, 00-662 Warsaw, Poland
\printead{e2}}

\address[c]{Microsoft Corporation, One Microsoft Way, Redmond, WA 98052, USA \printead{e3}}

\runauthor{K. Burdzy et al.}

\affiliation{University of Washington, Microsoft and Warsaw University of Technology}

\end{aug}

\begin{abstract}
We study solutions to the stochastic fixed point equation $X\stackrel{d}{=}AX+B$ when the coefficients are nonnegative and $B$
is an ``inverse exponential decay'' (\IGFT) random variable. 
We provide theorems on the left tail of $X$ which complement  well-known tail results of Kesten and Goldie. We generalize our results to ARMA processes with nonnegative coefficients whose noise terms are from the \IGFT{} class.
We describe the lower envelope for these ARMA processes.
\end{abstract}

\begin{keyword}
\kwd{Tail estimates} \kwd{inverse-gamma distribution} \kwd{stochastic fixed point equation} \kwd{iterated random sequences} \kwd{ARMA models} \kwd{time series}
\end{keyword}

\end{frontmatter}

\section{Introduction}

The paper is devoted to studying properties, especially left tails, of positive random variables that arise in several closely related contexts---stochastic fixed point equations, ARMA models,  and iterated random functions.

For a two-dimensional random vector $(A,B)$, an independent random variable $X$ is said to satisfy the stochastic fixed point equation 
if
\begin{equation}\label{stochasticFixedPointEquation}
X\stackrel{d}{=}AX+B. 
\end{equation}
The behavior of the solution, especially the left and right tails, has been extensively studied. A classical result (\cite{Kesten,GoldieAAP}) 
says that under some assumptions on $(A,B)$, for some $\alpha, C_-, C_+ >0$,
\begin{equation}\label{tailEstimates}
\P(X>x)\sim C_+x^{-\alpha}\quad \textrm{and}\quad \P(X<-x)\sim C_-x^{-\alpha}, 
\end{equation}
as $x\to\infty$. (See the precise statement in Theorem \ref{thm:KestenGoldieResult}. Here  $\sim$ means that the ratio of the two quantities converges to 1.)
An excellent review of  the subject can be found in a recent book \cite{X=AX+B}. 

It can be shown that 
if $A$ and $B$ are nonnegative random variables then the solution $X$ to \eqref{stochasticFixedPointEquation} is also a nonnegative random variable.
Under this extra assumption on $(A,B)$, the first estimate in $(\ref{tailEstimates})$ is still meaningful and informative. But the second one is not because for $x>0$ we have 
$\P(X<-x)=0$. 
It is natural to ask for a meaningful estimate for the left tail under these circumstances.  We will examine the behavior of $\P(X<x)$ as $x\to 0^+$. 
This question does not seem to be addressed anywhere in the literature; in particular,  it does not seem to be examined in \cite{X=AX+B}.

The motivation for the present paper comes from a project on a ``Fleming-Viot'' type process defined in \cite{BHM}. We will explain in Section \ref{sec:example} how the problem arises in the setting of \cite{extinctionOfFlemingViot}.

\subsection{Review of the main results}
This paper revolves around \IGFT$^\rho_L(\lambda)$ random variables defined as follows.

\begin{defin}\label{j15.3}
 We will say that a nonnegative random variable $X$ has an inverse exponential decay of the left tail with index $\rho>0$ if
\begin{equation}
\lim_{x\to 0^+}x^{\rho}L(x)\log \P(X<x)=-\lambda,\label{eq:definitionOfLambda0}
\end{equation}
for a slowly varying function $L$ at 0 and $\lambda\geq 0$.
We will call such a random variable \IGFT$^\rho_L(\lambda)$.
In the case $\lim_{x\to0^+}L(x)=1$, we will write \IGFT$_1^\rho(\lambda)$.
\end{defin}

The best known \IGFT$^\rho_L(\lambda)$ distributions  are called ``inverse-gamma;'' in this case, $\rho=1$
and $L (x) \equiv 1$ (see Definition \ref{n16.1}). 

In Section \ref{sec:seriesConvergence} (see especially  Theorems \ref{prop:firstMomentCondition} and  \ref{prop:expectationOfLogarithm}),
we will find  conditions  for a sequence $(X_i)$ of independent $\IGFT^\rho_L(\lambda_i)$-random variables  so that the series
$\sum_{i=1}^{\infty}\alpha_iX_i$
is an \IGFT$^\rho_L(\Lambda)$-random variable with  
$$\Lambda=\left(\sum_{i=1}^{\infty}\alpha_i^{\rho/(1+\rho)}\lambda_i^{1/(1+\rho)}\right)^{1+\rho}.$$

Consider an ARMA series of the form
$$X_n=\sum_{i=1}^p\phi_i X_{n-i}+B_n+\sum_{j=1}^q\theta_jB_{n-j}$$
with positive coefficients $\phi_i$ and $\theta_j$ and initial value 0. Assume that $(B_i)$ are i.i.d. $\IGFT^\rho_1(\lambda)$-random variables.
We will give  conditions (see Theorem \ref{thm:lowerEnvelope2}) so that $X_n$ converge to an $\IGFT^\rho_1(\Lambda)$-random variable and 
$$\liminf_{n\to \infty} (\log n)^{1/\rho} X_n =\Lambda^{1/\rho}>0,\ \text{ a.s.},$$
where $\Lambda$ is an explicit function of $\lambda, \rho$ and the coefficients of the recursion.

We will also study the stochastic fixed point equation 
$X\stackrel{d}{=}AX+B$ 
where the vector $(A,B)$ is independent of $X$, $B$ is an \IGFT$^\rho_L(\Lambda)$-random variable, and $A$ and $B$ are nonnegative and positively quadrant dependent (see Theorem \ref{thm:seriesConvegenceABIndependent}).
If $A$ and $B$ are not positively quadrant dependent, we will prove by example that Theorem \ref{thm:seriesConvegenceABIndependent}) (ii) need not be satisfied (see Section \ref{sec:example}).

\subsection{Organization of the paper}
In Section \ref{sec:IGFT}, we introduce \IGFT{}  random variables
and we prove that this class is closed under addition of finitely many independent summands. 
Section \ref{sec:seriesConvergence} is devoted to infinite series  of independent \IGFT{} random variables---we show that the sum may or may not be \IGFT.
In Section \ref{sec:AREquation}, we discuss the autoregressive equation, i.e.,  the fixed point equation with the multiplicative coefficient that is a constant. 
In Section \ref{sec:ARMA} we expand our results to ARMA models with positive coefficients and the noise from the \IGFT{} class.
In Section \ref{sec:RandomCoefficientIndependent}, we give estimates for  left tails of solutions to the fixed point equation when the coefficients are positively quadrant dependent random variables. In Section \ref{sec:example} we show that if the coefficients are not positively quadrant dependent then these results no longer hold and the analysis  is more demanding.

\section{Preliminaries}

We will  write $a^+ = \max(0,a)$ for any real $a$.

We will use the convention that for any sequence $(d_n)$ and $i>j$, $\sum_{n=i}^{j} d_n = 0$ and $\prod_{n=i}^{j} d_n = 1$.

Recall that the essential infimum of a random variable $A$ is defined as follows,
\begin{equation}\label{n8.1}
 \essinf(A)=  \sup\{x\in \R : \P(A< x)=0\} .
\end{equation}

If $\lim_{x\to 0^+} f(x)/g(x) =1$ then we will write $f(x)\sim g(x)$. The same notation will be used if the limit holds when $x\to \infty$.

\begin{defin} (See \cite[(1.2.1),  Sects. 1.4.1-1.4.2]{regularVariation}.)
A function $f:(0,\infty)\to (0,\infty)$ is called  slowly varying at 0 if for all $a>0$ we have 
$\lim_{x\to0^+} f(ax)/f(x) =1$. A function $f:(0,\infty)\to (0,\infty)$ is called  regularly varying  of index $\rho$ at 0 if for all $a>0$ we have 
$\lim_{x\to0^+} f(ax)/f(x) =a^\rho$.
A function $f$ is called regularly varying of index $\rho$ at infinity if $x\mapsto 1/f(1/x)$ is a regularly varying function of index $\rho$ at $0$.
\end{defin}

\begin{lemma} (See \cite[(1.2.1),  Thm 1.4.1]{regularVariation}.)
 A positive function $f$ is regularly varying  of index $\rho$ at 0 if and only if $f(x) = x^\rho L(x)$ for some  slowly varying function $L$ at 0.
\end{lemma}

\begin{defin} \label{n3.4}
(See \cite[(1.5.10)]{regularVariation}.)
If $f(x)$ is defined and locally bounded on some interval $(0,a]$ and $\lim_{x\to 0^+} f(x) = \infty$ then the generalized inverse of $f$ is defined by
\begin{align*}
f^\la(y) = \sup\{x>0: f(x) > y\}.
\end{align*}

If $f(x)$ is strictly positive on some interval $(0,a]$ and $\lim_{x\to 0^+} f(x) = 0$ then the generalized inverse of $f$ is defined by
\begin{align*}
f^\la(y) = \inf\{x>0: f(x) > y\}.
\end{align*}

\end{defin}

\begin{lemma}\label{assyptoicInverseLemma}
If $\alpha>0$ and $f$ is $\alpha$-regularly varying at  $0$ then there exists 
a function $g$ which is
 $1/\alpha$-regularly varying  $0$ and such that 
\begin{equation}
f(g(x))\sim g(f(x))\sim x  \label{eq:aymptoticInverse}
\end{equation}
as $x\to  0^+$. The function $g$, called an asymptotic inverse of $f$, is determined up to asymptotic equivalence and one version of $g$ is $f^{\leftarrow}$.  
\end{lemma}
\begin{proof}
The proof is routine so it is left to the reader. See \cite[Sect. 1.5.7]{regularVariation}, in particular Theorem 1.5.12.
\end{proof}

\section{Inverse exponential decay}\label{sec:IGFT}\rm

The definition of random variables with inverse exponential decay of the left tail is inspired, in part, by 
inverse gamma distributions. These are  used in Bayesian statistics (see  \cite{HoffBayes}). One way to define inverse gamma distributions is by saying that the reciprocal of a random variable with a gamma distribution has the inverse gamma distribution. A more direct definition follows.

\begin{defin}\label{n16.1}
For a positive random variable $X$ we say it has the inverse gamma distribution with parameters $\alpha,\beta >0$ 
if its density function has the form 
$$f(x)=\frac{\beta^\alpha}{\Gamma(\alpha)}x^{-\alpha-1}e^{-\beta/x},$$
for $x\in (0,\infty)$.
\end{defin}

Recall Definition \ref{j15.3}.
The notation $L, L_1$, etc., will be used exclusively for slowly varying functions at 0 unless stated otherwise.

\begin{lemma}\label{lemma:rho1LessRho2}
Suppose that $\rho_1<\rho_2$. If $X$ is an $\IED^{\rho_1}_{L_1}(\lambda)$-random variable then $X$ is $\IED^{\rho_2}_{L_2}(0)$-random variable 
for every slowly varying function $L_2$ at 0. 
\end{lemma}
\begin{proof}
The proof is routine and left to the reader.
\end{proof}

\begin{lemma}\label{propo:limitsEquality}
For any non-negative random variable $X$, 
\begin{equation}
\lim_{x\to 0^+}x^\rho L(x)\log \P(X<x)=\lim_{x\to 0^+}x^\rho L(x)\log \P(X\leq x) ,\label{eq:definitionOfLambdaInTwoWays}
\end{equation}
in the sense that if one of the  limits exists then the other one exists as well and they are equal.
\end{lemma}

\begin{proof}
For all $x>0$ and $\eps\in(0,1)$, 
\begin{align}\label{n3.1}
\P(X\leq x(1-\varepsilon)) \leq
\P(X<x) \leq \P(X\leq x)
\leq \P(X<x(1+\varepsilon)).
\end{align}
Therefore
\begin{equation} \label{limsupParameterLessLiminfParameter}
 \limsup_{x\to 0^+}x^{\rho}L(x)\log \P(X<x) \leq \liminf_{x\to 0^+} x^{\rho}L(x)\log \P(X\leq x).
\end{equation}

Assume that the limit $\lim_{x\to 0^+}x^{\rho}L(x)\log \P(X<x)=-\lambda$ exists. In view of the last inequality in \eqref{n3.1},  for any  $\varepsilon >0$,
\begin{align*}
&\limsup_{x\to 0^+} x^\rho L(x)\log \P(X\leq x) \\
&\leq 
\limsup_{x\to 0^+}\frac{L(x)}{L(x(1+\varepsilon))}\cdot \frac 1 {(1+\eps)^\rho}
\cdot\left( [x(1+\varepsilon)]^\rho L(x(1+\varepsilon))\log \P(X<x(1+\varepsilon)) \right)\\
&= 
\limsup_{x\to 0^+}\frac{L(x)}{L(x(1+\varepsilon))}\cdot \frac 1 {(1+\eps)^\rho}
\cdot(-\lambda)
=
-\frac{\lambda}{(1+\varepsilon)^\rho}.
\end{align*}
 Letting $\varepsilon\to 0^+$ and combining the resulting inequality with 
$(\ref{limsupParameterLessLiminfParameter})$ yields $(\ref{eq:definitionOfLambdaInTwoWays})$.

In the case when $\lim_{x\to 0^+}x^\rho L(x)\log \P(X\leq x)$ exists,
 a similar argument, based on  $(\ref{limsupParameterLessLiminfParameter})$ and the 
 first inequality in \eqref{n3.1}
 proves $(\ref{eq:definitionOfLambdaInTwoWays})$. 
\end{proof}

\example{We will show that the positive limit in $(\ref{eq:definitionOfLambda0})$ might not exist for any fixed $\rho$ and $L$. It is easy to see that there exists a c.d.f. $F$ with the property that
\begin{align}\label{n3.2}
F\left( 2^{-(3k+1)}\right)=e^{-2^{3k+1}}\quad \textrm{and}\quad F\left( 2^{-(3k+2)}\right)=e^{-2^{3k+3}}
\end{align}
for $k=1,2,\dots$, because $F$ restricted to the arguments listed in \eqref{n3.2} is increasing.  
If $X$ is a  random variable with c.d.f. $F$ then
\begin{align}\label{j15.1}
\lim_{k\to\infty}2^{-(3k+1)}\log\P\left(X\leq 2^{-(3k+1)}\right)&=-1, \\  \lim_{k\to\infty}2^{-(3k+2)}\log\P\left(X\leq 2^{-(3k+2)}\right)&=-2.\label{j15.2}
\end{align}
Assume that there exist $\rho$, a function $L$ slowly varying at $0$, and $\lambda>0$ such that $\log P(X<x)\sim -\lambda (x^\rho L(x))^{-1}$ as $x\to 0^+$.
Then \eqref{j15.1} shows that $\rho=1$.
Lemma \ref{propo:limitsEquality} and \eqref{j15.1}-\eqref{j15.2} imply that
\begin{align*}
\frac12 &=\lim_{k\to \infty}\frac{2^{-(3k+1)}\log\P\left(X\leq 2^{-(3k+1)}\right)}{2^{-(3k+2)}\log\P\left(X\leq 2^{-(3k+2)}\right)}
=\frac12 \lim_{k\to \infty}\frac{\log\P\left(X\leq 2\cdot 2^{-(3k+2)}\right)}{\log\P\left(X\leq 2^{-(3k+2)}\right)}\\
&=\frac12 \lim_{k\to \infty}\frac{2^{-(3k+2)} L( 2^{-(3k+2)})}
{2\cdot 2^{-(3k+2)} L(2\cdot 2^{-(3k+2)})}
=\frac14.
\end{align*}
This contradiction proves our claim.
}

\lemma{\label{propo:inverseGamaDistribTail}If $X$ has inverse gamma distribution with parameters $\alpha, \beta>0$ (see Definition \eqref{n16.1}), then $X$ is an \IGFT$^1_1(\beta)$-random variable.}

\begin{proof}
Consider any $\eps\in(0,1)$.
If $X$ is a random variable with the inverse gamma distribution with parameters $\alpha, \beta>0$ then for any $\eps\in(0,1)$ and sufficiently small $x>0$,
\begin{align*}
 \P(X<x)&\geq  \int_{(1-\varepsilon)x}^x\frac{\beta^\alpha}{\Gamma(\alpha)}e^{-\beta/t}t^{-\alpha-1}dt \\
 &\geq \frac{\beta^\alpha}{\Gamma(\alpha)}\exp\left(-\frac{\beta}{(1-\varepsilon)x}\right)((1-\varepsilon)x)^{-\alpha-1}\cdot\varepsilon x.
\end{align*}
This implies that $\liminf_{x\to 0^+}x\log \P(X<x)\geq -\beta/(1-\varepsilon)$, for all $\varepsilon \in (0,1)$. Hence, $\liminf_{x\to 0^+}x\log \P(X<x)\geq -\beta$.

On the other hand, 
\begin{align*}
\P(X<x)&
= \int_0^x \frac{\beta^\alpha}{\Gamma(\alpha)}e^{-\beta/t}t^{-\alpha-1}dt 
= \int_0^x \frac{\beta^\alpha}{\Gamma(\alpha)}
e^{-\beta (1-\varepsilon)/t}e^{-\beta\eps/t}t^{-\alpha-1}dt \\
&
\leq e^{-\beta (1-\varepsilon)/x} \int_0^x 
\frac{\beta^\alpha}{\Gamma(\alpha)}
e^{-\beta\eps/t}t^{-\alpha-1}dt
\leq e^{-\beta (1-\varepsilon)/x}\int_{0}^\infty\frac{\beta^\alpha}{\Gamma(\alpha)}e^{-\beta \varepsilon/t}t^{-\alpha-1}dt.
\end{align*}
Since the last integral is finite and independent of $x$, we have 
\begin{align*}
\limsup_{x\to 0^+}x\log \P(X<x)\leq -\beta (1-\varepsilon)
\end{align*}
 for all $\varepsilon \in (0,1)$. Hence, $\limsup_{x\to 0^+}x\log \P(X<x)\leq -\beta$. 
\end{proof}

\rm
The following two propositions are elementary so their proofs are left to the reader.

\begin{propo} \label{n3.3}
 Suppose that  $X$ is an $\IED^\rho_L(\lambda)$ random variable, $\gamma>0$ and $L_1(x) \equiv L(x^{1/\gamma})$. Then $X^\gamma$ is an $\IED^{\rho/\gamma}_{L_1}(\lambda)$ random variable.
 
 In particular, if $X$ is an $\IED^\rho_1(\lambda)$ random variable and $\gamma>0$ then $X^\gamma$ is an $\IED^{\rho/\gamma}_1(\lambda)$ random variable.

\end{propo}

\begin{propo} \label{propo:closureOnMultiplication} 
If  $X$ is an $\IED^\rho_L(\lambda)$ random variable and $\alpha>0$ then $\alpha X$ is an $\IED^{\rho}_{L}(\alpha^\rho\lambda)$ random variable.
\end{propo}

\rm 
\example{\label{remark:inverseVariablesTail}Note that if $X$ is a nonnegative random variable with the property 
$$\lim_{x\to\infty}\frac{\log \P(X>x)}{x^\rho}=-\lambda\leq 0,$$
then $X^{-1}$ is an $\IGFT^\rho_1(\lambda)$-random variable. We give two natural examples.
\begin{enumerate}[(a)]
 \item If $X$ has the exponential distribution with parameter $\lambda$ (i.e., its mean is $1/\lambda$), then $X^{-1}$ is an \IGFT$^1_1(\lambda)$-random variable.
 \item If $X$ has the normal distribution with mean $\mu$ and variance $\sigma^2$ then $X^{-2}$ is an \IGFT$^1_1((2\sigma^2)^{-1})$-random variable. It follows from Proposition \ref{n3.3} that $|X|^{-1}= (X^{-2})^{\frac{1}{2}}$ is an \IGFT$^2_1((2\sigma^2)^{-1})$-random variable.
\end{enumerate}}

\medskip
The following is an alternative characterization of  \IGFT-random variables that we will use often.

\begin{lemma}\label{propo:characterization}
A positive random variable $X$ 
is \IGFT$^\rho_L(\lambda)$ 
if and only if for every $\delta>0$ there exist $x_0>0$, 
$c_{\delta}>0$ and $C_{\delta}>0$ such that for all $x\in (0,x_0)$,
\begin{equation}
c_{\delta}\exp\left(-\frac{\lambda+\delta}{x^{\rho}L(x)}\right)\leq \P(X<x)\leq \P(X\leq x) \leq C_{\delta}\exp\left(-\frac{\lambda(1-\delta)}{x^{\rho}L(x)}\right). \label{eq:inverseGammaLikeCharacterization}
\end{equation} 
\end{lemma}
\begin{proof}
The claim follows easily from the definition \eqref{eq:definitionOfLambda0}.
\end{proof}

\rm
The following theorem, the springboard for the rest of the paper, says that \IGFT{} random variables form a closed class under natural operations.

\thm{\label{sum2IED}
Suppose that $X_1$ and $X_2$ are independent random variables and $X_k$ is \IGFT$^\rho_L(\lambda_k)$, for $k=1,2$.
Then $X_1+X_2$ is an \IGFT$^\rho_{L}\left([\lambda_1^{1/(1+\rho)}+\lambda_2^{1/(1+\rho)}]^{1+\rho}\right)$ random variable.

 }\rm
 
 \medskip
 
Before proving the theorem we recall without proof  de Bruijn's Tauberian Theorem
(see \cite[Thm. 4.12.9]{regularVariation}). See Definition \ref{n3.4} for the generalized inverse $f^{\leftarrow}$.

\thm{\label{thm:Bruijin}
Suppose that $\mu$ is a measure on $(0,\infty)$ with the finite Laplace transform
$$M(z):=\int_0^{\infty}e^{-z x}\mu(dx)<\infty,\quad \textrm{for all}\ z>0.$$
Suppose that $\alpha<0$ and $ \phi(x)$ is a regularly varying function with index  $\alpha$ at 0. Then, 
\begin{align}\label{n4.1}
\lim_{x\to 0^+} \phi^{\leftarrow}(1/x) \log \mu((0,x])=-\lambda
\end{align}
if and only if 
\begin{align}\label{n4.2}
-\log M(z) \sim 
(1-\alpha) (- \lambda/\alpha)^{\alpha/(\alpha -1)}
/\psi^{\leftarrow}(z),\ \textrm{as} \ z\to \infty,
\end{align}
where $\psi(z)=\phi(z)/z$.}

\begin{proof}[Proof of Theorem \ref{sum2IED}]
Let $ \P_{X_j}$ denote the distribution of $X_j$ and
$$M_j(z)=\int_0^{\infty}e^{-zx}\P_{X_j}(dx), \quad j=1,2.$$
It is clear that for $M_j(z)\leq 1<\infty$ for all $z>0$ and $j=1,2$. 

If $L_2$ is slowly varying at 0 then $\phi(x) := (1/x)^{1/\rho} L_2(1/x)$ is regularly varying with index $-1/\rho$ at  infinity. 
Arguments analogous to those in  \cite[Sect. 1.5.7]{regularVariation}
(see also Lemma \ref{assyptoicInverseLemma})
show that we can choose $L_2$ so that $\phi^\la (1/x) \sim x^\rho L(x)$ when $x\to 0^+$.
With this choice of $\phi$,
the assumption that  $X_j$ is \IGFT$^\rho_L(\lambda_j)$ matches \eqref{n4.1}, so \eqref{n4.2} holds, i.e., for $j=1,2$,
\begin{align*}
-\log M_j(z) \sim 
(1+1/\rho) ( \lambda_j\rho)^{1/(1+\rho )}
/\psi^{\leftarrow}(z),\ \textrm{as} \ z\to \infty,
\end{align*}
where $\psi(z)=\phi(z)/z$. Since $X_1$ and $X_2$ are independent,
\begin{align*}
 -\log M_{X_1+X_2}(z) &= -\log M_{X_1}(z)-\log M_{X_2}(z)\\
&\sim (1+1/\rho) \rho^{1/(1+\rho )}
 \left[\left(\lambda_1^{1/(1+\rho)}+\lambda_2^{1/(1+\rho)}\right)^{1+\rho}\right]^{1/(1+\rho)}
/\psi^{\leftarrow}(z),
\end{align*}
as $z\to \infty$. 
The proof is completed by
reversing our argument, using Theorem \ref{thm:Bruijin} and applying it to $X_1+X_2$.
\end{proof}

\propo{\label{thm:sumOfVariablesTail}
Suppose that $\rho_1<\rho_2<\ldots<\rho_n$ and for $i=1,\ldots, n$ and $j=1,\ldots, m_i$, $X_{ij}$ is  $\IED^{\rho_i}_{L_i}(\lambda_{ij})$, and all these random variables are independent. Then 
$$S:=\sum_{i=1}^n\sum_{j=1}^{m_i} \alpha_{ij}X_{ij}$$
is an $\IED^{\rho_n}_{L_n}(\Lambda)$-random variable, where
$$\Lambda=\left(\sum_{j=1}^{m_n}\alpha_{nj}^{\rho_n/(1+\rho_n)}\lambda_{nj}^{1/(1+\rho_n)}\right)^{1+\rho_n}.$$
}
\begin{proof}
Lemma \ref{lemma:rho1LessRho2} and Proposition \ref{propo:closureOnMultiplication} imply that
$$\lim_{x\to 0^+}x^{\rho_n} L_n(x)\log\P(\alpha_{ij}X_{ij}<x)=\left\{ 
\begin{array}{cl}
 0, & i\neq n;\\
 -\alpha_{ij}^{\rho_n}\lambda_{ij},& i=n.
\end{array}
\right.$$
The proposition follows from Theorem \ref{sum2IED} and induction.
\end{proof}

\section{Convergence of infinite \IGFT-series}\label{sec:seriesConvergence}\rm

It is a natural question whether Proposition \ref{thm:sumOfVariablesTail}
holds for infinite series with independent \IGFT{} summands. The short answer is ``no'' but ``yes'' under extra assumptions.

\propo{\label{corol:limsup}
Suppose that $(X_i)_{i\geq 1}$ are independent and $X_i$ is $\IGFT^\rho_L(\lambda_i)$, for $i\geq 1$.
 Let $S=\sum_{i\geq 1} X_i$ and $\Lambda =\left(\sum_{i\geq 1}\lambda_i^{1/(1+\rho)}\right)^{1+\rho}$. Then 
\begin{equation}
 \limsup_{x\to 0^+}x^\rho L(x) \log \P(S<x)\leq  -\Lambda. \label{eq:inequalityTail}
\end{equation}
}
\begin{proof}
Since $X_i$'s are nonnegative random variables, for every $n$ we have
$\P(S<x)\leq \P\left(\sum_{i=1}^nX_i<x\right)$.
This and Proposition \ref{thm:sumOfVariablesTail} imply
\begin{align*}
\limsup_{x\to 0^+}x^\rho L(x)\log \P(S<x) &\leq \limsup_{x\to 0^+}x^\rho L(x)\log \P\left(\sum_{i=1}^nX_i<x\right) \\
&=-\left(\sum_{i=1}^n\lambda_i^{1/(1+\rho)}\right)^{1+\rho}.
\end{align*}
If we let $n\to \infty$, the claim follows.
\end{proof}

\example{Recall the notation and assumptions from Proposition \ref{corol:limsup}.
Two examples given below  show that, in general, 
\eqref{eq:inequalityTail} cannot be strengthened to equality. The first example is a little bit more elegant than the second one. But $S\equiv \infty$ in the first example, suggesting that divergence of the sum  $\sum_{i\geq 1} X_i$ is the only possible obstacle to having equality in \eqref{eq:inequalityTail}. For this reason we  present another example with $S\leq 1$, a.s.

 (a) Suppose that $X_i$'s have inverse gamma distributions with parameters $\alpha_i=1/4$ and $\beta_i =1/i^4$, for $i\geq 1$ (see Definition \ref{n16.1}).
 According to Lemma \ref{propo:inverseGamaDistribTail}, $X_i$ is an \IGFT$^1_1(\beta_i)$-random variable, for every $i\geq 1$. Let $\rho=1$ and $\lambda_i=\beta_i=1/i^4$ for all $i$.
Then, for $i\geq 1$, 
\begin{align*}
 \P(X_i>1)&= \int_1^{\infty} \frac{1}{i\Gamma(1/4)}x^{-1/4-1}e^{-1/(i^4x)}dx
 \geq \frac{1}{i}\int_1^{\infty} \frac{1}{\Gamma(1/4)}x^{-1/4-1}e^{-1/x}dx.
\end{align*}
Hence, 
$$\sum_{i\geq 1} \P(X_i>1)=\infty.$$
Therefore, by the Borel-Cantelli lemma, $X_i>1$ for infinitely many $i$'s. It follows that $S=\infty$, a.s., regardless of the fact that 
$$
\Lambda :=\left(\sum_{i\geq 1}\lambda_i^{1/(1+1)}\right)^{1+1}
=
\left(\sum_{i=1}^{\infty}\frac{1}{i^2}\right)^2<\infty.$$

(b) 
For some $\tau_i>0$, $i\geq 1$, to be specified later, we give
 $X_i$'s the following cumulative distribution functions, 
\begin{equation*}
 F_i(x)= \left\{ \begin{array}{cl}
          0& \textrm{for} \ x\leq 0;\\
                 e^{1-\tau_i}e^{-1/(2^i x)}&\textrm{for} \ 0<x<2^{-i};\\
           1& \textrm{for  } 2^{-i}\leq x.
                                      \end{array}
 \right.
\end{equation*}
Note that each $X_i$ is an \IGFT$^1_1(2^{-i})$-random variable.
Let $\lambda_i = 2^{-i}$ and note that $S:=\sum_{i=1}^{\infty}X_i\leq 1$, a.s., because   $X_i\leq 2^{-i}$, a.s., for all $i$.
Although $\sqrt{\Lambda}=\sum_{i=1}^{\infty}\sqrt{\lambda_i}<\infty$, we will show that
$S$ 
is not an \IGFT$^1_1(\Lambda)$-random variable. Consider an integer $n\geq 1$ and let $x=2^{-(n+1/2)}$.
Hence $2^{-(n+1)}<x<2^{-n}$. Since each $X_i$ is bounded by $2^{-i}$, we have 
$$\{S<x\}\subset \bigcap_{i=1}^{n}\left\{X_i<2^{-i}\right\}.$$ Hence, using the fact that $\P(X_i<2^{-i})=e^{-\tau_i}$, 
 we have
\begin{align}\label{n21.5}
x\log \P(S<x)\leq x \sum_{i=1}^{n}\log \P(X_i<2^{-i})= -x\sum_{i=1}^{n}\tau_i.
\end{align}
 If we choose $\tau_i =c2^{i-1/2}$
then
\begin{align*}
x\sum_{i=1}^{n}\tau_i = x\sum_{i=1}^{n} c2^{i-1/2}
=   2^{-(n+1/2)} c 2^{-1/2} (2^{n+1} -1) = c (1-2^{-n-1}) . 
\end{align*}
This and \eqref{n21.5} imply that
$\liminf_{x\to 0^+}x\log \P(S<x)\leq  -c$. If we set $-c< -\Lambda$, then $\liminf_{x\to 0^+}x\log \P(S<x)<  -\Lambda$.
} 

\bigskip

We will  give sufficient conditions for the equality in 
$(\ref{eq:inequalityTail})$ in Theorems \ref{prop:firstMomentCondition} and  \ref{prop:expectationOfLogarithm}. The main technical part of the proof is contained in the following lemma.

\lemma{\label{thm:mainResult1}

Suppose that $(X_i)_{i\geq 1}$ is a sequence of independent random variables such that
$$\lim_{x\to 0^+}x^\rho L(x)
\log \P(X_i<x)=-\lambda_i,$$
for a sequence $(\lambda_i)_{i\geq 1}$ of non-negative real numbers and a slowly varying function $L$ at 0.
Assume that $$\Lambda^{1/(1+\rho)} :=\sum_{i\geq 1}\lambda_i^{1/(1+\rho)}<\infty.$$
Suppose that a random variable $B$ satisfies the following conditions. 
\begin{enumerate}[(i)]
\item $B$ has the property 
$$\lim_{x\to 0^+}x^\rho L(x)\log \P(B<x)=-1.$$
 \item $B$ is stochastically greater than  $X_i/\lambda_i^{1/\rho}$
for each $i\geq 1$,  i.e., for all $x\in \R$ and $i\geq 1$,
$$\P(B\leq x)\leq \P\left(\frac{X_i}{\lambda_i^{1/\rho}}\leq x\right).$$
\item There exist positive real numbers $(\gamma_i)_{i\geq 1}$ such that
$\sum_{i=1}^\infty\gamma_i=1$, 
$\sum_{i=1}^{\infty}\lambda_i/\gamma_i^\rho<\infty$, and
\begin{equation}
\lim_{x\to 0^+}x^\rho L(x)\sum_{i=1}^{\infty} \P\left(B\geq \frac{\gamma_i}{\lambda_i^{1/\rho}}x\right)=0.
 \label{eq:dominatonByDeterministicSequence2}
\end{equation} 
\end{enumerate}
Then $\sum_{i=1}^{\infty} X_i$ converges  to an a.s. finite random variable $S$ satisfying
\begin{equation*}
 \lim_{x\to 0^+}x^\rho L(x)\log \P(S<x)=  -\Lambda. 
\end{equation*}}

\begin{proof}
Since $B$ is stochastically greater than $X_i/\lambda_i^{1/\rho}$, we have 
\begin{align*}
\P(X_i\geq \gamma_ix)=
\P(X_i/\lambda_i^{1/\rho}\geq \gamma_ix/\lambda_i^{1/\rho})\leq \P(B\geq \gamma_ix/\lambda_i^{1/\rho}).
\end{align*}
It follows from \eqref{eq:dominatonByDeterministicSequence2}
 that for  small $x>0$ we have 
\begin{align}\label{j13.1}
\sum_{i=1}^{\infty} \P\left(B\geq  \gamma_ix/\lambda_i^{1/\rho}\right)<\infty.
\end{align}
Hence, for small $x>0$,  $\sum_{i=1}^{\infty} \P\left(X_i\geq \gamma_ix\right)<\infty$. By the Borel-Cantelli lemma,
 the sequence $(X_i)$ is eventually dominated by 
$\left(\gamma_ix\right)$ a.s. Recall that $\sum_{i=1}^\infty\gamma_i=1$ to see that $\sum_{n=1}^{\infty}X_n$
converges to an a.s. finite random variable $S$.

Let $S_n=\sum_{k=1}^nX_k$. We have
\begin{align*}
 \{S<x\}&\supset \left\{S_n<x\sum_{k=1}^n\gamma_k\right\}
 \cap \bigcap_{k=n+1}^{\infty} \left\{X_{k}<\gamma_{k}x\right\}
\end{align*}
so, by assumption (ii),
\begin{align}
 \P(S<x)&\geq \P\left(S_n<x\sum_{k=1}^n\gamma_k\right) \prod_{k=n+1}^{\infty} \P\left(X_{k}<\gamma_{k}x\right)\nonumber\\
&\geq \P\left(S_n<x\sum_{k=1}^n\gamma_k\right) \prod_{k=n+1}^{\infty} \P\left(B<\frac{\gamma_{k}}{\lambda_{k}^{1/\rho}}x\right). \label{eq:probabilityProduct1}
\end{align} 

By assumption (i), for $\varepsilon>0$
there exists $x_0>0$ such that 
\begin{align*}
\P(B<x)\geq \exp(-2(x^\rho L(x))^{-1}),
\end{align*}
for $0< x < x_0$.
Let 
\begin{align*}
F_n(x_0)&=\left\{k\geq n+1:x_0>\frac{\gamma_kx}{\lambda_k^{1/\rho}}\right\},
\qquad
G_n(x_0)=\left\{k\geq n+1:x_0\leq\frac{\gamma_kx}{\lambda_k^{1/\rho}}\right\}.
\end{align*}

Using this notation, we can break the last product in $(\ref{eq:probabilityProduct1})$ as follows:
\begin{align}
&\nonumber \prod_{k=n+1}^{\infty}\P\left(B<\frac{\gamma_{k}}{\lambda_{k}^{1/\rho}}x\right) = \prod_{k\in F_n(x_0)}\P\left(B<\frac{\gamma_{k}}{\lambda_{k}^{1/\rho}}x\right)\prod_{k\in G_n(x_0)}\P\left(B<\frac{\gamma_{k}}{\lambda_{k}^{1/\rho}}x\right)\\
\label{eq:probabilityProduct2}&\geq \exp\left[-\frac{2}{x^\rho}\sum_{k\in F_n(x_0)}\frac{\lambda_k}{\gamma_k^\rho}L\left(\frac{\gamma_kx}{\lambda_k^{1/\rho}}\right)^{-1}\right]\prod_{k\in G_n(x_0)}\left[1-\P\left(B\geq\frac{\gamma_{k}}{\lambda_{k}^{1/\rho}}x\right)\right].
\end{align}
By the definition of $G_n(x_0)$, for all $k\in G_n(x_0)$ we have 
\begin{align*}
\P\left(B\geq\frac{\gamma_{k}}{\lambda_{k}^{1/\rho}}x\right)\leq \P(B\geq x_0)<1.
\end{align*}
Standard calculus arguments show that there exists $m>1$ such that 
$-ma\leq \log(1-a)$ for $0<a\leq\P(B<x_0)$. Hence, by $(\ref{eq:probabilityProduct2})$,
\begin{align}
\nonumber \prod_{k=n+1}^{\infty}&\P\left(B<\frac{\gamma_{k}}{\lambda_{k}^{1/\rho}}x\right) \\
 &\geq \exp\left(-\frac{2}{x^\rho}\sum_{k\in F_n(x_0)}\frac{\lambda_k}{\gamma_k^\rho}L\left(\frac{\gamma_kx}{\lambda_k^{1/\rho}}\right)^{-1}-m\sum_{k\in G_n(x_0)}\P\left(B\geq\frac{\gamma_{k}}{\lambda_{k}^{1/\rho}}x\right)\right)\nonumber\\
&\geq \exp\left(-\frac{2}{x^\rho}\sum_{k=n+1}^\infty\frac{\lambda_k}{\gamma_k^\rho}L\left(\frac{\gamma_kx}{\lambda_k^{1/\rho}}\right)^{-1}-m\sum_{k=1}^\infty\P\left(B\geq\frac{\gamma_{k}}{\lambda_{k}^{1/\rho}}x\right)\right).\label{eq:probabilityProduct3}
\end{align}
Applying Proposition \ref{thm:sumOfVariablesTail} to  the first factor on the last line of \eqref{eq:probabilityProduct1}, 
\begin{align}\notag
\liminf_{x\to 0^+}  x^\rho L(x) \log\P\left(S_n<x\sum_{k=1}^n\gamma_k\right) 
& =\liminf_{x\to 0^+}x^\rho L(x) \log \P\left(\left(\sum_{k=1}^n\gamma_k\right)^{-1}S_n<x\right) \\
 &\geq -
\left(\sum_{i\geq 1}\lambda_i^{1/(1+\rho)}\right)^{1+\rho} \left(\sum_{k=1}^n\gamma_k\right)^{-\rho}\label{o25.1}
\end{align} 
The estimate  $(\ref{eq:probabilityProduct3})$ and the assumption
\eqref{eq:dominatonByDeterministicSequence2}
yield for the second factor in $(\ref{eq:probabilityProduct1})$,
\begin{align}\notag
\liminf_{x\to 0^+}  x^\rho& L(x) \log\left(
\prod_{k=n+1}^{\infty}\P\left(B<\frac{\gamma_{k}}{\lambda_{k}^{1/\rho}}x\right) 
\right)\\
& \geq\liminf_{x\to 0^+}x^\rho L(x) 
\left(
-\frac{2}{x^\rho}\sum_{k=n+1}^\infty\frac{\lambda_k}{\gamma_k^\rho}L\left(\frac{\gamma_kx}{\lambda_k^{1/\rho}}\right)^{-1}-m\sum_{k=1}^\infty\P\left(B\geq\frac{\gamma_{k}}{\lambda_{k}^{1/\rho}}x\right)
\right).\label{eq:boundRefreeFix}
\end{align} 
Set $f(z)=z^{\rho}L(z)$. By \cite[Thm 1.5.2]{regularVariation}, 
\begin{align}\label{o30.1}
\lim_{z\to0} f(z)/f(z/h)= h^\rho,
\end{align}
 uniformly in $h$, on each fixed interval of the form $[b,\infty)$. Since $\gamma_i/\lambda_i^{1/\rho}\to \infty$ as $i\to \infty$, there is $b>0$ such that all the values of this sequence are in $[b,\infty)$. We apply \eqref{o30.1} to see that for some $C_1< \infty$ and $x_1>0$, for all $x\in(0, x_1)$ and $i\geq 1$,
\begin{equation}
\frac{\lambda_i}{\gamma_i^\rho}L(x)L\left(\frac{\gamma_ix}{\lambda_i^{1/\rho}}\right)^{-1}
=
\frac{f(x)}{f\left(\gamma_ix/\lambda_i^{1/\rho}\right)}
\leq C_1 \frac{\lambda_i}{\gamma_i^\rho}. \label{eq:uniformBound}
\end{equation}
This, \eqref{j13.1} and \eqref{eq:boundRefreeFix} imply that
\begin{equation}
\liminf_{x\to 0^+}  x^\rho L(x) \log\left(
\prod_{k=n+1}^{\infty}\P\left(B<\frac{\gamma_{k}}{\lambda_{k}^{1/\rho}}x\right) 
\right)
\geq  -2C_1\sum_{i=n+1}^{\infty}\frac{\lambda_i}{\gamma_i^\rho}.\label{o25.2} 
\end{equation}

Combining \eqref{eq:probabilityProduct1}, \eqref{o25.1} and \eqref{o25.2} gives
\begin{align}\label{n21.10}
\liminf_{x\to 0^+}x^\rho L(x) \log \P(S<x) \geq-
\left(\sum_{i\geq 1}\lambda_i^{1/(1+\rho)}\right)^{1+\rho} \left(\sum_{k=1}^n\gamma_k\right)^{-\rho} -2C_1\sum_{i=n+1}^{\infty}\frac{\lambda_i}{\gamma_i^\rho}.
\end{align}
Recall that we have assumed that $\Lambda^{1/(1+\rho)} =\sum_{i\geq 1}\lambda_i^{1/(1+\rho)}<\infty$, $\sum_{i=1}^\infty\gamma_i=1$,
 and $\sum_{i=1}^{\infty}\lambda_i/\gamma_i^\rho<\infty$.
Thus, when we let $n\to \infty$ in \eqref{n21.10}, we obtain
\begin{align*}
\liminf_{x\to 0^+}x^\rho L(x)\log \P(S<x) \geq-\Lambda.
\end{align*}
The opposite inequality follows from Proposition \ref{corol:limsup}.
\end{proof}

\thm{\label{prop:firstMomentCondition}
Suppose that $(B_i)_{i\geq 1}$ is a sequence of i.i.d. \IGFT$_L^\rho(1)$-random variables 
satisfying
\begin{equation}
 \limsup_{x\to \infty}x^\rho \P(B_1\geq x)<\infty,
  \label{eq:almostPthMoment}
\end{equation}
and $(\alpha_i)_{i\geq 1}$ is a sequence of strictly positive real numbers 
satisfying
$\sum_{i=1}^{\infty} \alpha_i^{\rho/(1+\rho)}<\infty$.
Then the  series 
$\sum_{i=1}^{\infty}\alpha_iB_i$
converges a.s. to an $\IED_L^\rho(\Lambda)$ random variable, where
$$\Lambda= \left(\sum_{i=1}^{\infty}\alpha_i^{\rho/(1+\rho)}\right)^{1+\rho}.$$
}

\begin{remark}\label{n4.3}
Condition \eqref{eq:almostPthMoment} implies that 
$\E[B_1^\tau]<\infty$ for $\tau \in [0,\rho)$. If $\E[B_1^\rho]<\infty$ then  \eqref{eq:almostPthMoment} is satisfied. 
\end{remark}

\begin{proof}[Proof of Theorem \ref{prop:firstMomentCondition}]

We will apply Lemma \ref{thm:mainResult1}. Let $X_i = \alpha_iB_i$ and $\lambda_i = \alpha_i^\rho$, for $i\geq 1$. By Proposition \ref{propo:closureOnMultiplication}, $B_i$ is
$\IED^\rho_L(\alpha_i^\rho )$-random variable, for $i\geq 1$. Let $B$ be  distributed as $B_1$. It is easy to see that assumptions (i) and (ii) 
of Lemma \ref{thm:mainResult1} are satisfied by $X_i$'s and $B$.

Let $c=\left(\sum_{i=1}^{\infty}\alpha_i^{\rho/(1+\rho)}\right)^{-1}$ 
and $\gamma_i=c\alpha_i^{\rho/(1+\rho)}$ for $i\geq 1$. 
Then $\sum_{i=1}^{\infty}\gamma_i=1$ and 
\begin{align}\label{o30.3}
\sum_{i=1}^{\infty}\frac{\lambda_i}{\gamma_i^\rho}=\sum_{i=1}^\infty\frac{\alpha_i^{\rho}}{c^\rho \alpha_i^{\rho^2/(1+\rho)}}=\frac{1}{c^\rho}\sum_{i=1}^\infty \alpha^{\rho/(1+\rho)}= c^{-1-\rho}<\infty,
\end{align}
so two conditions listed in assumption (iii) of Lemma \ref{thm:mainResult1} are satisfied. 
It remains to verify 
\eqref{eq:dominatonByDeterministicSequence2}.

Without loss of generality we can assume  that
$\lim_{x\to \infty}L(x)=1$. Then  \eqref{eq:almostPthMoment} is equivalent to $\limsup_{x\to \infty}x^\rho L(x) \P(B\geq x)<\infty$.
We have  $\lim_{x\to 0^+}x^\rho L(x) \P(B\geq x)=0$
because $\rho >0$ and $L$ is slowly varying at 0.
The two conditions imply that there exists $C>0$ such that $\P(B\geq x) \leq C x^{-\rho} L(x)^{-1}$ for all $x>0$. In particular, we have for all $x>0$,
\begin{align}\label{n22.1}
\P\left(B\geq \frac{\gamma_i}{\lambda_i^{1/\rho}}x\right)\leq C x^{-\rho}\frac{\lambda_i}{\gamma_i^\rho}L\left(\frac{\gamma_ix}{\lambda_i^{1/\rho}}
\right)^{-1}.
\end{align}

For every fixed $i$, $\lim_{x\to 0^+}x^\rho L(x)\P\left(B\geq \frac{\gamma_i}{\lambda_i^{1/\rho}}x\right)=0$, 
so for every fixed $n$,
\begin{equation}
 \limsup_{x\to 0^+}x^\rho L(x)\sum_{i=1}^\infty\P\left(B\geq \frac{\gamma_i}{\lambda_i^{1/\rho}}x\right)=\limsup_{x\to 0^+}x^\rho L(x)\sum_{i=n}^\infty\P\left(B\geq \frac{\gamma_i}{\lambda_i^{1/\rho}}x\right). \label{eq:bound} 
\end{equation}
By \eqref{n22.1}, 
\begin{align}\label{o30.2}
x^\rho L(x)\sum_{i=n}^\infty\P\left(B\geq \frac{\gamma_i}{\lambda_i^{1/\rho}}x\right) \leq C\sum_{i=n}^\infty 
\frac{\lambda_i}{\gamma_i^\rho}L(x)L\left(\frac{\gamma_ix}{\lambda_i^{1/\rho}}\right)^{-1}.
\end{align}
 This, \eqref{eq:uniformBound}, \eqref{eq:bound},  \eqref{o30.2} and \eqref{o30.3} imply that
 $$\limsup_{x\to 0^+}x^\rho L(x)\sum_{i=1}^\infty\P\left(B\geq \frac{\gamma_i}{\lambda_i^{1/\rho}}x\right)
 \leq  C C_1 \sum_{i=n}^\infty \frac{\lambda_i}{\gamma_i^\rho} < \infty.  $$
Letting $n\to \infty$, we obtain  
$$\lim_{x\to 0^+}x^\rho L(x)\sum_{i=n}^\infty\P\left(B\geq \frac{\gamma_i}{\lambda_i^{1/\rho}}x\right)=0.$$
We see that \eqref{eq:dominatonByDeterministicSequence2} holds and, therefore, the theorem follows from Lemma $\ref{thm:mainResult1}$.
\end{proof}\rm

\begin{example}
Suppose that $(B_i)_{i\geq 1}$ are i.i.d. $\IED^\rho_1(1)$-random variables with finite $\rho$-th moment.
Theorem \ref{prop:firstMomentCondition} and Remark \ref{n4.3} imply that
for all $c,\varepsilon >0$, the series $\sum_{i=1}^\infty  c i^{-(1+\varepsilon)(1+\rho)/\rho}B_i$ 
converges a.s. The limit is an $\IGFT^\rho_1(\Lambda)$-random variable with  parameter $\Lambda =c^{\rho}\left(\sum_{i=1}^{\infty} i^{-1-\varepsilon}\right)^{1+\rho}$.
\end{example}

The following theorem shows that  if the parameters $\alpha_i$ decrease at a geometric rate then we can weaken the condition on the moments of $B_i$
and obtain the same conclusion as in Theorem \ref{prop:firstMomentCondition}.
 
\thm{\label{prop:expectationOfLogarithm}

Suppose that $(B_i)_{i\geq 1}$ is a sequence of i.i.d. \IGFT$^\rho_L(1)$-random variables 
satisfying $\E[\log^+B_i]<\infty$.
For any sequence of strictly positive real numbers $(\alpha_i)_{i\geq 1}$ with the property
$\limsup_{i\to\infty}\sqrt[i]{\alpha_i}=\kappa\in(0,1)$, the  series 
$\sum_{i=1}^{\infty}\alpha_iB_i$
converges a.s., and the limit is an \IGFT$^\rho_L(\Lambda)$-random variable with
$$ \Lambda=\left(\sum_{i=1}^{\infty}\alpha_i^{\rho/(1+\rho)}\right)^{1+\rho}.$$
}
\begin{proof}

We will apply Lemma \ref{thm:mainResult1}. Let $X_i = \alpha_iB_i$ and $\lambda_i =\alpha_i^\rho$ for $i\geq 1$, and let $B$ be a random variable with the same distribution as $B_1$. It is easy to see that assumptions (i) and (ii) of Lemma \ref{thm:mainResult1} are satisfied.
It only remains to show that condition (iii) is satisfied.

Pick $\zeta \in(\kappa, 1)$, and set $\gamma_i =\zeta^{i-1}(1-\zeta)$. 
Note that $\sum_{i=1}^{\infty}\gamma_i=1$. It is  clear that there
exist $c>0$ and $\kappa_1$ such that $0<\kappa<\kappa_1<\zeta$
and $0\leq \lambda_i^{1/\rho}=\alpha_i \leq c\kappa_1^i$ for all $i\geq 1$. 
Hence, 
$$\sum_{i=1}^{\infty}\frac{\lambda_i}{\gamma_i^\rho}\leq  c_1 \frac{\zeta^\rho}{(1-\zeta)^\rho} \sum_{i=1}^{\infty} \left(\kappa_1/\zeta\right)^{i \rho}<\infty,$$ 
so two conditions listed in assumption (iii) of Lemma \ref{thm:mainResult1} are satisfied. It remains to verify 
\eqref{eq:dominatonByDeterministicSequence2}.

We will use the following well known inequality, saying that for any positive random variable $X$ we have 
\begin{equation}
 \sum_{k=1}^{\infty}\P(X\geq k)\leq \E X +1. \label{eq:expectationLowerBound}
\end{equation}
The above inequality is used to justify the second inequality below,
\begin{align*}
 \sum_{i=1}^\infty\P\left(B\geq \frac{\gamma_i}{\lambda_i^{1/\rho}}x\right)
&\leq\sum_{i=1}^\infty\P\left(B\geq \frac{1-\zeta}{c\zeta }\left(\frac{\zeta}{\kappa_1}\right)^{i}x\right)\\
&=\sum_{i=1}^\infty\P\left(\log B -\log x+\log(c\zeta(1-\zeta)^{-1})\geq i\log \left(\zeta/\kappa_1\right)\right)\\
&\leq \frac{\E\left(\log B -\log x+\log(c\zeta(1-\zeta)^{-1})\right)^+}{\log \left(\zeta/\kappa_1\right)}+1\\
&\leq \frac{\E[\log^+ B] +|\log x|+|\log(c\zeta(1-\zeta)^{-1}|}{\log \left(\zeta/\kappa_1\right)}+1.
\end{align*}
The last estimate implies that
$\lim_{x\to0^+}x^\rho L(x)\sum_{i=1}^\infty\P\left(B\geq \gamma_ix/\lambda_i\right) =0$ because we have assumed that $\E[\log^+B_i]<\infty$.
We conclude that assumption (iii) of Lemma \ref{thm:mainResult1} holds. 
\end{proof}

\section{Autoregressive equation}\label{sec:AREquation}\rm

We will consider solutions to the autoregressive equation, a simple ARMA model, in this section. More general ARMA models will be considered in subsequent sections.

We start by recalling a known result. We would like to point out that random variables $A$ and $B$ need not be independent for the following to hold.

\thm{\label{thm:fixedPointRepresentation2}
If $\E[\log |A|]<0$ and $\E[\log^+|B|]<\infty$ then  $(\ref{stochasticFixedPointEquation})$ has a unique solution.
Suppose that $(A_i,B_i)_{i\geq 1}$ are i.i.d.  two-dimensional vectors distributed as $(A,B)$.
\begin{enumerate}[(a)]
 \item The distribution of the solution to $(\ref{stochasticFixedPointEquation})$ is the stationary distribution for the Markov chain given by
\begin{equation}
 X_n=A_nX_{n-1}+B_n. \label{eq:sequence2}
\end{equation}

 \item The series 
\begin{equation}
 S=\sum_{i=1}^\infty \left(\prod_{j=1}^{i-1}A_j\right)B_i \label{eq:solution:series2}
\end{equation}
converges a.s. and the distribution of the limit is the same as that of the solution to $(\ref{stochasticFixedPointEquation})$.
\end{enumerate}
}
\begin{proof}
By  \cite[Theorem 2.1.3]{X=AX+B} and  \cite[Theorem 2.1]{diaconisFreedman},  the sequence $(\ref{eq:sequence2})$ has a
unique ergodic invariant stationary distribution. Moreover, $(\ref{eq:solution:series2})$ is a representation of that distribution.
By \cite[Lemma 2.2.7]{X=AX+B}, this distribution is the unique solution to the fixed point equation  $(\ref{stochasticFixedPointEquation})$.
\end{proof}
\rm

In the rest of this section we will take a look at the nonnegative solution to the autoregressive equation 
\begin{equation}
X\stackrel{d}{=}rX+B, \label{eq:AR1}
\end{equation}
where $0<r<1$.

\corol{\label{thm:fixedPointRepresentation}
If $\E[\log^+B]<\infty$ then  $(\ref{eq:AR1})$ has a unique solution.
Suppose that $(B_i)$ are i.i.d.  random variables distributed as $B$.
\begin{enumerate}[(a)]
 \item The distribution of the solution to \eqref{eq:AR1} is the stationary distribution of the Markov chain given by
\begin{equation}
 X_n=rX_{n-1}+B_n. \label{eq:sequence}
\end{equation}

 \item The series 
\begin{equation}
 \sum_{i=1}^\infty r^{i-1}B_i \label{eq:solution:series}
\end{equation}
converges a.s. and the distribution of the limit is the same as that of the solution to $(\ref{eq:AR1})$.
\end{enumerate}
}
\begin{proof} The corollary follows from Theorem \ref{thm:fixedPointRepresentation2}.
\end{proof}

\corol{\label{tail:theorem}
If $B$ is an \IGFT$^\rho_L(\lambda)$-random variable such that $\E[\log^+B]<\infty$, and $X$ is the solution to  \eqref{eq:AR1} then $X$ is an \IGFT$^\rho_L(\Lambda)$-random variable with
\begin{equation}\label{n5.1}
 \Lambda =  \frac{\lambda}{(1-r^{\rho/(1+\rho)})^{1+\rho}}.
\end{equation}}
\rm

\begin{proof} The corollary follows from Theorem \ref{prop:expectationOfLogarithm} and
Corollary \ref{thm:fixedPointRepresentation} (b).
\end{proof}

The following result is a special case of Theorem \ref{thm:lowerEnvelope2} so we leave it without proof.

\propo{\label{thm:lowerEnvelope}If
$B$ is \IGFT$^\rho_1(\lambda)$,
  $\E[(\log^+B)^s]<\infty$ for all $s>0$, $\Lambda$ is defined in \eqref{n5.1}, and $X_n$'s are defined  in $(\ref{eq:sequence})$ then 
$$\liminf_{n\to \infty} (\log (n))^{1/\rho} X_n =\Lambda^{1/\rho}, \ \text{ a.s.}$$
}\rm

This proposition (for case $\rho=1$) is illustrated in Figure \ref{exampleLowerEnvelopeAR1_1}. Note that the result holds under mild assumptions on the right tail of $B$.

\begin{figure}
\begin{center}
 \includegraphics[height=6.5cm]{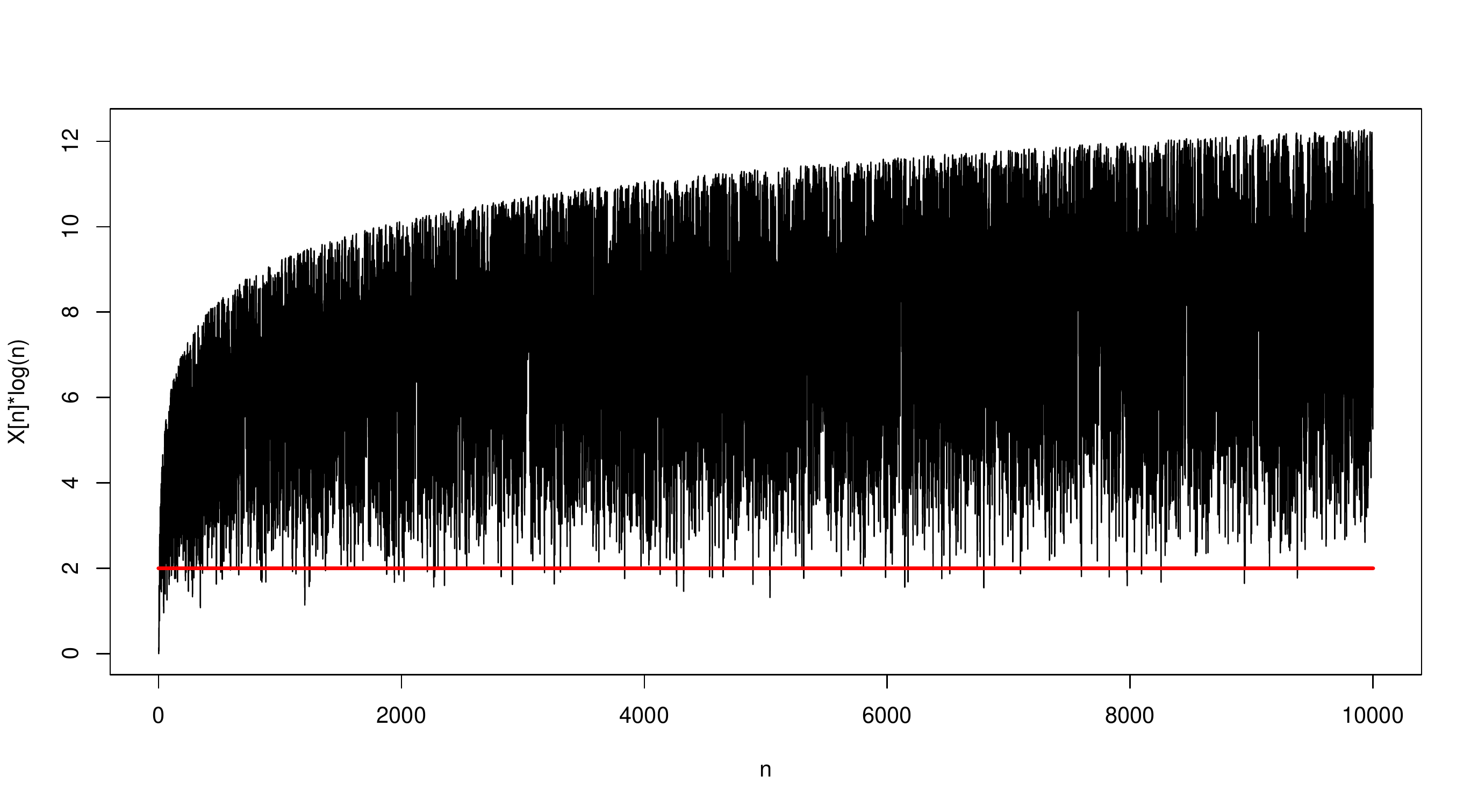}
\caption{The figure  shows the the  graph of $(\log(n)X_n)_{n\geq 1}$, where $(X_n)_{n\geq 1}$ is the ARMA process $X_{n+1}=\frac{1}{4}X_{n}+B_{n+1}$, where $B_n=\min(E_n^{-1},1)$, $E_n$'s are i.i.d
exponential with parameter $1/2$  and $X_0=0$. It follows from Proposition \ref{thm:lowerEnvelope} that $\Lambda =2$. This is visible in the graph as the black line segments occasionally reach to the horizontal line at level 2.}\label{exampleLowerEnvelopeAR1_1}
\end{center}
\end{figure}

\section{ARMA models with \IGFT{} noise}\label{sec:ARMA}

This section is devoted to autoregressive moving average (ARMA) models more general than those in the previous section.

\defin{\label{def:ARMA(p,q)like}An ARMA$(p,q)$ 
sequence has the form
\begin{equation}
X_n=\sum_{i=1}^p\phi_iX_{n-i}+B_n+\sum_{j=1}^q\theta_jB_{n-j},
\label{eq:arma(p,q)}
\end{equation}
where $(\phi_i)_{i=1,\dots,p}$ and $(\theta_j)_{j=1,\dots,q}$ are positive 
constants and $(B_i)_{i\geq1}$ are i.i.d.} 

\medskip

Our analysis of ARMA models will be based on \cite[Chap. 3]{BrockwellDavis}.
Using the notation from  Definition \ref{def:ARMA(p,q)like}, we define complex polynomials $\Phi$ and $\Theta$ by
$$\Phi(z) =1-\phi_1 z-\phi_2 z^2-\ldots -\phi_pz^p \quad \textrm{and}\quad
\Theta(z) = 1+\theta_1z+\ldots + \theta_q z^q.
$$

\thm{\label{thm:polynomialInterpretation}
Suppose that $(B_n)_{n\geq 1}$ are i.i.d.~\IGFT$^\rho_L(\lambda)$ random variables
and we have $\E[\log^+B_1]<\infty$.
Assume that $\Phi(z)\neq 0$ for $|z|\leq 1$ and $\Phi$
and $\Theta$ have no common roots. Then 
$\Psi(z) :=\Theta(z)/\Phi(z)$ is analytic on a neighborhood of the unit disc $\{|z|\leq 1\}$, and its Taylor series, i.e.,
\begin{equation}
\Psi(z)=\sum_{k=0}^{\infty}\psi_kz^k, \label{eq:TayloExpantion}
\end{equation}
has positive coefficients. 
\begin{enumerate}[(i)]
	\item Random variables $X_n$ in \eqref{eq:arma(p,q)} can be represented as
	\begin{equation}
			X_n=\sum_{k=0}^n \psi_k B_{n-k}.	\label{eq:recursionSolution}
	\end{equation}
	
	\item Each $X_n$ is an \IGFT$^\rho_L(\Lambda_n)$-random variable with  
 $$\Lambda_n =\lambda\left(\sum_{k=0}^n \psi_k^{\rho/(1+\rho)}\right)^{1+\rho}.$$

	\item When $n\to \infty$, 
$X_n\stackrel{d}{\to} X:=\sum_{k=0}^{\infty}\psi_k B_k$, and the limit is a finite \IGFT$^\rho_L(\Lambda)$-random  variable with  
$$\Lambda=\lambda\left(\sum_{k=0}^\infty \psi_k^{\rho/(1+\rho)}\right)^{1+\rho}.$$
\end{enumerate}}
\begin{proof}
There exists $\varepsilon>0$ such that $|\phi_1z+\phi_2z^2+\ldots+\phi_pz^p|<1$ for $|z|<\varepsilon$.
For such $z$,
$$\Psi(z)= \sum_{j=0}^{\infty}(\phi_1z+\phi_2z^2+\ldots+\phi_pz^p)^j(1+\theta_1z+\ldots + \theta_q z^q).$$
It is evident from this formula that for $|z|<\varepsilon$, $\Psi(z)$ can be represented as a series with positive coefficients. By the uniqueness of Taylor series,
all  $\psi_k$'s are positive. 
 
The function $\Psi$ is analytic on a disc around $0$ whose radius is greater than 1. 
Hence, the Taylor series of $\Psi$ around $0$ has a convergence radius $R> 1$. By the Cauchy-Hadamard formula,
$\limsup_{k\to\infty}|\psi_k|^{1/k}=R^{-1}<1$. Therefore, there exist $C>0$ and $0<\beta<1$ such that $\psi_k<C\beta^k$ for $k\geq 0$. This implies that both series $\sum_{k\geq1}\psi_k$ and $\sum_{k\geq1}\sqrt{\psi_k}$ converge. 

Part (i) follows from \cite[Thm. 3.1.1]{BrockwellDavis}.
Part (ii) follows from Proposition \ref{thm:sumOfVariablesTail}.
Part (iii) follows from Proposition \ref{thm:sumOfVariablesTail} and Theorem \ref{prop:expectationOfLogarithm}.
\end{proof}\rm

We will now  prove a  generalization of Proposition \ref{thm:lowerEnvelope}.

\thm{\label{thm:lowerEnvelope2}
Consider an ARMA sequence $(\ref{eq:arma(p,q)})$ satisfying the assumptions of Theorem \ref{thm:polynomialInterpretation} and recall the notation from \eqref{eq:TayloExpantion}. If $\E[(\log^+B)^r]<\infty$ for all $r>0$ then  
$$\liminf_{n\to \infty}  \frac{X_n}{g(1/\log n)} =\Lambda^{1/\rho} := \lambda^{1/\rho}\left(\sum_{k=0}^\infty \psi_k^{\rho/(1+\rho)}\right)^{(1+\rho)/\rho},\ \text{ a.s.},$$
where $g$ is the generalized inverse of the function $x\mapsto x^{\rho} L(x)$ at $0$.
}\rm

The proof will be preceded by a few lemmas.

\lemma{\label{lemma:easySideOfBorelCantelli}

(i) For every $\varepsilon>0$, the events
$$\left\{X_n\leq g\left(\frac{\Lambda}{(1+\varepsilon)\log n}\right)\right\}$$
happen finitely often, a.s.

(ii)  We have
$$\liminf_{n\to\infty}\frac{X_n}{g(1/\log n)}  \geq \Lambda^{1/\rho},\ \text{ a.s.}$$
}
\begin{proof}
(i) For any $\eps>0$
there exists $\delta\in(0,1)$ such that $\gamma := (1-\delta)(1+\varepsilon/2)>1$. 
By Theorem \ref{thm:polynomialInterpretation},
there exist $C_\delta$, $x_0$ and $n_0$ such that $$\P(X_{n_0}\leq x)\leq C_\delta \exp\left(-\frac{\Lambda(1-\delta)}{x^\rho L(x)}\right)$$ for all $x\in (0, x_0)$.

Random variables $B_k$ are i.i.d., so \eqref{eq:recursionSolution} implies that
	\begin{equation*}
			X_n=\sum_{k=0}^n \psi_k B_{n-k}
\stackrel{d}= \sum_{k=0}^n \psi_k B_{k}.
	\end{equation*}
Since $\psi_k$'s and $B_k$'s are nonnegative, it follows 
that $X_{n+1}$ stochastically majorizes $X_n$ for all $n$. Hence, $n\mapsto \P(X_{n}\leq x)$ is a non-increasing sequence and, therefore, for $n\geq n_0$ and $x\in (0,x_0)$,
$$\P(X_n\leq x)\leq C_\delta \exp\left(-\frac{\Lambda(1-\delta)}{x^\rho L(x)}\right).$$
It follows that for large $n$,
$$\P\left(X_n\leq g\left(\frac{\Lambda}{(1+\varepsilon)\log n}\right)\right)
\leq C_\delta e^{-(1+\varepsilon/2)(1-\delta)\log n}=C_{\delta}n^{-\gamma}.
$$
Hence, 
$$\sum_{n=1}^{\infty}\P\left(X_n\leq g\left(\frac{\Lambda}{(1+\varepsilon)\log n}\right)\right)<\infty,$$
and the claim follows by the Borel-Cantelli lemma.

(ii) It follows from part (i) that for every $\eps >0$,
$$\liminf_{n\to\infty}g\left(\frac{\Lambda}{(1+\varepsilon)\log n}\right)^{-1}X_n
\geq 1, \quad \text{a.s.}$$
By Lemma \ref{assyptoicInverseLemma}, $g$ is $1/\rho$-regularly varying at $0$. Hence, 
 $$g\left(\frac{\Lambda}{(1+\varepsilon)\log n}\right)\sim \left(\frac{\Lambda}{1+\varepsilon}\right)^{1/\rho} g(1/\log n).$$
Therefore,
$$\liminf_{n\to\infty}\frac{X_n}{g(1/\log n)}  \geq \Lambda^{1/\rho}/(1+\varepsilon)^{1/\rho}.$$
Part (ii) follows by letting $\eps\to0$.
\end{proof}\rm

It will be convenient to use the following notation, reminiscent of $(\ref{eq:recursionSolution})$, 
\begin{align}\label{d29.1}
X_n^m=\sum_{k=0}^{n-m-1} \psi_k B_{n-k}.
\end{align}
Recall that $\lfloor a \rfloor$ denotes the largest integer less than or equal to $a$.

\lemma{\label{lemma:artifficialSequence}

Fix  $\varepsilon>0$ and suppose that $1<1+\delta< \sqrt{1+\eps}$ and $\alpha \in (1, 1+\delta)$. Events 
$$\left\{X_{\lfloor n^\alpha \rfloor}^{\lfloor (n-1)^\alpha \rfloor}\leq g\left(\frac{\Lambda(1+\varepsilon)}{\log n^{\alpha}}\right)\right\}$$
happen infinitely often a.s. }
\begin{proof}
Note that the  random variables  $X_{\lfloor n^\alpha \rfloor}^{\lfloor (n-1)^\alpha \rfloor}$, $n\geq 2$, are jointly independent.

The random variable
$S:=\sum_{k=0}^{\infty}\psi_k B_k$ stochastically majorizes every 
$X_{\lfloor n^\alpha \rfloor}^{\lfloor (n-1)^\alpha \rfloor}$.
By Theorem \ref{thm:polynomialInterpretation} (iii) and 
Lemma \ref{propo:characterization}, for large $n$,
 \begin{align*}
  \P&\left(X_{\lfloor n^\alpha \rfloor}^{\lfloor (n-1)^\alpha \rfloor}\leq g\left(\frac{\Lambda(1+\varepsilon)}{\log n^{\alpha}}\right)\right)
  \geq\P\left(S\leq g\left(\frac{\Lambda(1+\varepsilon)}{\log n^{\alpha}}\right)\right)\\
&\geq 
c_\delta \exp\left(
- \frac{\Lambda (1+\delta)\log n^{\alpha}}{\Lambda(1+\varepsilon)}
\right)
=
c_\delta n^{-\alpha(1+\delta)/(1+\varepsilon)}\geq n^{-\alpha/(1+\delta)} .
 \end{align*}
It follows that
 \begin{align*}
  &\sum_{n=2}^{\infty}\P\left(X_{\lfloor n^\alpha \rfloor}^{\lfloor (n-1)^\alpha \rfloor}\leq g\left(\frac{\Lambda(1+\varepsilon)}{\log n^{\alpha}}\right)\right) = \infty,
 \end{align*}
and, therefore, the claim follows by the Borel-Cantelli lemma.
\end{proof}\rm

\lemma{\label{lem:tailOfTheartificialSequence}
If $\alpha >1$ then, 
$$\lim_{n\to \infty}
 \frac{X_{\lfloor n^\alpha \rfloor}-X_{\lfloor n^\alpha \rfloor}^{\lfloor (n-1)^\alpha \rfloor}}{g(1/\log \lfloor n^\alpha \rfloor)}
 = 0,\quad \text {  a.s.  }$$
}

\begin{proof}

We have
\begin{align*}
X_{\lfloor n^\alpha \rfloor}-X_{\lfloor n^\alpha \rfloor}^{\lfloor (n-1)^\alpha \rfloor}
= \sum_{k=\lfloor n^\alpha \rfloor-\lfloor (n-1)^\alpha \rfloor}^{\lfloor n^\alpha \rfloor} \psi_k B_{\lfloor n^\alpha \rfloor-k}
=\sum_{j=0}^{\lfloor (n-1)^\alpha \rfloor} \psi_{\lfloor n^\alpha \rfloor-\lfloor (n-1)^\alpha \rfloor+j} B_{\lfloor (n-1)^\alpha \rfloor-j} .
\end{align*}
This and the estimate $\psi_k\leq C\beta^k$ from the proof of Theorem \ref{thm:polynomialInterpretation} yield
\begin{align}
X_{\lfloor n^\alpha \rfloor}-X_{\lfloor n^\alpha \rfloor}^{\lfloor (n-1)^\alpha \rfloor}
&\leq \sum_{j=0}^{\lfloor (n-1)^\alpha \rfloor} C\beta^{\lfloor n^\alpha \rfloor-\lfloor (n-1)^\alpha \rfloor+j} B_{\lfloor (n-1)^\alpha \rfloor-j} \label{eq:inequalityDiffReducedSequence} \\ 
\nonumber&=C\beta^{\lfloor n^\alpha \rfloor-\lfloor (n-1)^\alpha \rfloor}\sum_{j=0}^{\lfloor (n-1)^\alpha \rfloor} \beta^{j} B_{\lfloor (n-1)^\alpha \rfloor-j}.
\end{align}
Recall that $\beta\in(0,1)$ and $\alpha>1$.
It is not hard to show that there exists $c>0$ such that $\lfloor n^\alpha \rfloor-\lfloor (n-1)^\alpha \rfloor\geq cn^{\alpha-1}$ for large  $n$, so
\begin{align}
X_{\lfloor n^\alpha \rfloor}-X_{\lfloor n^\alpha \rfloor}^{\lfloor (n-1)^\alpha \rfloor}
\nonumber&\leq C\beta^{cn^{\alpha-1}}\sum_{j=0}^{\lfloor (n-1)^\alpha \rfloor} \beta^{j} B_{\lfloor (n-1)^\alpha \rfloor-j}\\ 
\nonumber&\leq  C\beta^{(c/2)n^{\alpha-1}}\sum_{j=0}^{\lfloor (n-1)^\alpha \rfloor} \beta^{j} \left(\beta^{(c/2)n^{\alpha-1}} B_{\lfloor (n-1)^\alpha \rfloor-j}\right)\\
\nonumber&\leq  C\beta^{(c/2)n^{\alpha-1}}\sum_{j=0}^{\lfloor (n-1)^\alpha \rfloor} \beta^{\lfloor (n-1)^\alpha \rfloor-j} \left(\beta^{(c/2)(n^{\alpha})^{(\alpha-1)/\alpha}} B_{j}\right)\\
&\leq  C\beta^{(c/2)n^{\alpha-1}}\sum_{j=0}^{\lfloor (n-1)^\alpha \rfloor} \beta^{\lfloor (n-1)^\alpha \rfloor-j} \left(\beta^{(c/2)j^{(\alpha-1/\alpha}} B_{j}\right).\label{n24.1}
\end{align}

We use the assumption that $\E[(\log^+ B)^{\alpha/(\alpha-1)}]<\infty$ and inequality $(\ref{eq:expectationLowerBound})$ to see that, for any $c_1>0$,
\begin{align*}
&\sum_{n=1}^{\infty}\P\left( B_n\geq  \beta^{-c_1n^{(\alpha-1)/\alpha}}\right)
=\sum_{n=1}^{\infty}\P\left(\log B_n\geq   c_1n^{(\alpha-1)/\alpha}\log \beta^{-1}\right)\\
&\leq \sum_{n=1}^{\infty}\P\left(\log^+ B_n\geq   c_1n^{(\alpha-1)/\alpha}\log \beta^{-1}\right)
= \sum_{n=1}^{\infty}\P\left(
\frac{ 1}{ c_1 \log \beta^{-1}} \log^+ B_n\geq   n^{(\alpha-1)/\alpha}\right)\\
&= \sum_{n=1}^{\infty}\P\left(\left(
\frac{ 1}{ c_1 \log \beta^{-1}}\right)^{\alpha/(\alpha-1)} 
\left(\log^+ B_n\right)^{\alpha/(\alpha-1)}\geq   n\right)\\
&\leq \E\left (\left(
\frac{ 1}{ c_1 \log \beta^{-1}}\right)^{\alpha/(\alpha-1)} 
\left(\log^+ B_n\right)^{\alpha/(\alpha-1)}\right)+1
<\infty.
\end{align*}
If we take $c_1=c/2$ then, by the Borel-Cantelli lemma, 
with probability 1,
\begin{align*}
 K:= \sup_{n\geq 1} \beta^{(c/2) n^{(\alpha-1)/\alpha}}B_n
< \infty.
\end{align*}
This and \eqref{n24.1} imply that 
\begin{align*}
 X_{\lfloor n^\alpha \rfloor}-X_{\lfloor n^\alpha \rfloor}^{\lfloor (n-1)^\alpha \rfloor} &\leq  C\beta^{(c/2)n^{\alpha-1}}\sum_{j=0}^\infty \beta^{j}  K
= C\beta^{(c/2)n^{\alpha-1}}\frac 1 {1-\beta} K.
\end{align*}
Thus, a.s.,
\begin{align*}
\limsup_{n\to \infty}
\frac{X_{\lfloor n^\alpha \rfloor}-X_{\lfloor n^\alpha \rfloor}^{\lfloor (n-1)^\alpha \rfloor}}{g(1/\log \lfloor n^\alpha \rfloor)} 
\leq
\limsup_{n\to \infty}
\frac{C\beta^{\frac{c}{2}n^{\alpha-1}}}{ g(1/\log \lfloor n^\alpha \rfloor)} \cdot \frac 1 {1-\beta} K =0.
\end{align*}
\end{proof}

\begin{proof}[Proof of Theorem \ref{thm:lowerEnvelope2}]
By Lemmas \ref{assyptoicInverseLemma},
\ref{lemma:artifficialSequence} and
 \ref{lem:tailOfTheartificialSequence},  for every $\varepsilon>0$,
\begin{align*}
\liminf_{n\to \infty}&\frac{ X_n}{g(1/\log n)}
\leq \liminf_{n\to \infty} \frac{X_{\lfloor n^\alpha \rfloor}}{g(1/\log  \lfloor n^\alpha \rfloor)}\\
&=
\liminf_{n\to \infty}
\left( \frac{X_{\lfloor n^\alpha \rfloor}-X_{\lfloor n^\alpha \rfloor}^{\lfloor (n-1)^\alpha \rfloor}}{g(1/\log  \lfloor n^\alpha \rfloor)}
+  \frac{X_{\lfloor n^\alpha \rfloor}^{\lfloor (n-1)^\alpha \rfloor}}{g(1/\log   \lfloor n^\alpha \rfloor)} \right)\\
&=
\liminf_{n\to \infty}
\left( \frac{X_{\lfloor n^\alpha \rfloor}-X_{\lfloor n^\alpha \rfloor}^{\lfloor (n-1)^\alpha \rfloor}}{g(1/\log  \lfloor n^\alpha \rfloor)}
+  \frac{X_{\lfloor n^\alpha \rfloor}^{\lfloor (n-1)^\alpha \rfloor}}
{g(\Lambda(1+\varepsilon)/\log   \lfloor n^\alpha \rfloor)} \Lambda^{1/\rho}(1+\varepsilon)^{1/\rho}\right)\\
&\leq \Lambda^{1/\rho}(1+\varepsilon)^{1/\rho}.
\end{align*}
Hence, $\liminf_{n\to \infty}  X_n/g(1/\log n) \leq \Lambda^{1/\rho}$, a.s. 
The theorem follows from this and Lemma \ref{lemma:easySideOfBorelCantelli} (ii). 
\end{proof}

\section{Random multiplicative coefficient}\label{sec:RandomCoefficientIndependent}\rm
So far, we only considered products of \IGFT\ random variables with  constants. In \eqref{eq:AR1}, the multiplicative coefficient 
in the stochastic fixed point equation was a constant. In this section  we will look into the case 
when these constants are replaced with nonnegative random variables independent of other random elements of the model.

In order to solve the stochastic fixed point equation $X\stackrel{d}{=}AX+B$
we will need an assumption on the form of dependence between  random variables $A$ and $B$. In this paper we will assume that $A$ and $B$ are positively quadrant dependent. 
This is a well known dependence condition,
used in various models in insurance and actuarial sciences. We start with the standard definition of positive quadrant dependence. 

\begin{defin}
We will call random variables $X$ and $Y$ positively quadrant dependent if
\begin{equation}
 \P(X>x,Y>y)\geq \P(X>x)\P(Y>y),\label{eq:positivelyQuadrantDependent}
\end{equation}
for all $x,y\in \R$. 
\end{defin}
\remark{Note that if two random variables are independent then they are also positively quadrant dependent.}

For the purposes of this paper the following characterization
of positive quadrant dependence
 will be more useful than the original definition.

\begin{lemma}\label{n11.1}
The random variables $X$ and $Y$ are positively quadrant dependent if and only if
\begin{equation}
 \P(X\leq x,Y\leq y)\geq \P(X\leq x)\P(Y \leq y)\label{eq:positivelyQuadrantDependentCharacterization}
\end{equation}
for all $x,y\in \R$. 
\end{lemma}
\begin{proof}
We add $-\P(X>x)$ to both sides
of \eqref{eq:positivelyQuadrantDependent} to obtain 
$$-\P(X>x,Y\leq y)\geq -\P(X>x)\P(Y\leq y).$$
We add $\P(Y\leq y)$ to  both sides of the last inequality to obtain \eqref{eq:positivelyQuadrantDependentCharacterization}.
This process can be reversed so \eqref{eq:positivelyQuadrantDependent} can be derived from \eqref{eq:positivelyQuadrantDependentCharacterization}.
\end{proof}

Recall Definition \ref{n8.1} of  essential infimum of a random variable.

\thm{\label{thm:IGFTofAX}
Suppose that $A$ is a nonnegative random variable and its essential infimum is equal to $a$. If $X$ is an $\IGFT^\rho_L(\lambda)$-random variable and $X$ and $A$ are positively quadrant dependent
then $AX$ is an \IGFT$^\rho_L(a^\rho\lambda)$-random variable.
 }
\begin{proof}
Since $a$ is the essential infimum of $A$, we have $a\leq A$, a.s., so
$$\P(AX<x)=\P(AX<x, a\leq A )\leq \P(aX<x).$$
Using Proposition \ref{propo:closureOnMultiplication}, 
$$\limsup_{x\to0^+}x^\rho L(x)\log \P(AX<x)\leq \lim_{x\to0^+}x^\rho L(x)\log \P(aX<x)=-a^\rho\lambda.$$

Let $\varepsilon >0$. The assumption that $A$ and $X$ are positively quadrant dependent implies that
\begin{align*}
 \P(AX<x) &\geq \P(AX<x, a\leq A\leq a+\varepsilon )
 \geq \P((a+\varepsilon)X<x, A\leq a+\varepsilon )\\
&\geq\P((a+\varepsilon)X<x)\P( A\leq a+\varepsilon ).
\end{align*}
By Proposition \ref{propo:closureOnMultiplication},
\begin{align*}
 &\liminf_{x\to0^+}x^\rho L(x)\log \P(AX<x)\\
&\geq \lim_{x\to0^+}x^\rho L(x)(\log \P((a+\varepsilon)X<x)+\log \P( a\leq A\leq a+\varepsilon ))\\ 
&=-(a+\varepsilon)^\rho\lambda.
\end{align*}
The proof is completed by letting $\varepsilon \to 0^+$.
\end{proof}

\corol{\label{corol:RandomCoefficientIndependentTail}
Suppose that independent random vectors $(A_i,X_i)$ are such that
for all $i=1,\ldots, n$,
\begin{enumerate}[(a)]
 \item $A_i$ and $X_i$ are nonnegative and positively quadrant dependent;
 \item  $X_i$ is an \IGFT$^\rho_L(\lambda_i)$-random variable.
\end{enumerate}
Then $A_1X_1+\ldots + A_nX_n$ is an \IGFT$^\rho_L(\Lambda)$-random variable with the parameter
$$\Lambda =\left((\essinf(A_1)^\rho\lambda_1)^{1/(1+\rho)}+\ldots + (\essinf(A_n)^\rho\lambda_n)^{1/(1+\rho)}\right)^{1+\rho}.$$
}\begin{proof}
  The corollary follows from Theorem \ref{thm:IGFTofAX} and Proposition \ref{thm:sumOfVariablesTail}.
 \end{proof}\rm

\rm

\thm{\label{thm:seriesConvegenceABIndependent}
Let $(A_i,B_i)_{i\geq 1}$ be an i.i.d. sequence of two-dimensional  vectors with the following properties. 
\begin{enumerate}[(i)]
 \item $A_1$ and $B_1$ are nonnegative and positively quadrant dependent.
 \item $\E[\log A_1]<0$ and $\E[\log^+B_1]<\infty$.
 \item $B_1$ is an \IGFT$^\rho_L(\lambda)$-random variable.
\end{enumerate}

(a) The series
$$\sum_{i=1}^\infty \left(\prod_{j=1}^{i-1}A_j\right)B_i$$
converges a.s. to a finite \IGFT$^\rho_L(\Lambda)$-random variable $S$, where
\begin{align}\label{n12.1}
\Lambda =\left(1-\essinf(A_1)^{\rho/(1+\rho)}\right)^{-1-\rho}
\lambda.
\end{align}

(b) The stochastic fixed point equation
$X\stackrel{d}{=}A_1 X+B_1$, where $X$ and $(A_1, B_1)$ are independent,
has a unique solution with the same distribution as that of $S$.
}\rm

We will need the following lemma.

\begin{lemma}\label{lemma:Tchen1980}
Assume that for all $a,b\in \R$,
$$\P(A\leq a, B\leq b)\geq \P(A'\leq a,B'\leq b), $$
with $A\stackrel{d}{=}A'$ and  $B\stackrel{d}{=}B'$. If $h(x,y)$ is bounded and $\frac{\partial^2}{\partial x\partial y}h\geq 0$
then $$\E[h(A,B)]\geq \E[h(A',B')].$$
\end{lemma}

\begin{proof}
The lemma is a special case of \cite[Thm. 2]{tchen1980}.
\end{proof}

\begin{proof}[Proof of Theorem \ref{thm:seriesConvegenceABIndependent}]
(a)
To simplify notation, let $(A,B)$ have the same distribution as $(A_1,B_1)$.
It follows from Theorem \ref{thm:fixedPointRepresentation2}  that $S$ is the solution to the stochastic 
fixed point equation $X\stackrel{d}{=}AX+B$.
We set $a =\essinf(A)$, $f_B(z)=-\log \E[e^{-zB}]$ and $f_S(z)=-\log \E[e^{-zS}]$. 
We have assumed that $\E[\log A]<0$ so $a\in[0,1)$.
By Theorem \ref{thm:Bruijin},
 $f_B$ is regularly varying at infinity with index $\rho/(1+\rho)$. By the same theorem, it will suffice to  show 
\begin{align}\label{n11.3}
\lim_{z\to\infty}\frac{f_S(z)}{f_B(z)}=\left(1-a^{\rho/(1+\rho)}\right)^{-1}.
\end{align}
If $S$ is independent of $(A,B)$ then $S\stackrel{d}{=}AS+B$ and, therefore,
\begin{align}
\label{eq:logLaplacianTransform}
 e^{-f_S(z)}=\E e^{-zS}=\E e^{-z(AS+B)}
\leq \E e^{-z(aS+B)} 
=  \E e^{-zaS} \E e^{-zB} 
= e^{ - f_S(az)-f_B(z)}, 
\end{align}
and
\begin{align}
\label{n11.2}
 e^{-f_S(z)}&=\E e^{-zS}=\E e^{-z(AS+B)}
=\E \left(\E\left( e^{-z(AS+B)} \mid A,B\right)\right)\\
&=\E \left(e^{-zB} \E\left( e^{-zAS}  \mid A,B\right)\right)
=\E \left(e^{-zB} e^{- f_S(zA)}\right)
=\E \left(e^{- f_S(zA)-zB}\right).\notag
\end{align}
It follows from \eqref{eq:logLaplacianTransform} that $f_S(z)\geq f_S(az)+f_B(z)$ and 
\begin{align*}
 \liminf_{z\to\infty}\frac{f_S(z)}{f_B(z)}&\geq \liminf_{z\to\infty} \frac{f_S(az)}{f_B(z)}+1
=\liminf_{z\to\infty} \frac{f_B(az)}{f_B(z)}\frac{f_S(az)}{f_B(az)}+1
\\
&=a^{\rho/(1+\rho)}\liminf_{z\to\infty} \frac{f_S(z)}{f_B(z)}+1,
\end{align*}
hence 
\begin{align}\label{n11.4}
\liminf_{z\to\infty} \frac{f_S(z)}{f_B(z)} \geq \left(1-a^{\rho/(1+\rho)}\right)^{-1}.
\end{align}

We will apply Lemma \ref{lemma:Tchen1980} to the function
\begin{align*}
h(x,y) = \exp(- f_S(xz) -yz) = \E \exp( - xzS - yz)
\end{align*}
and independent random variables $A'$ and $B'$ such that 
$A'\stackrel{d}{=}A$ and  $B'\stackrel{d}{=}B$.
Since  $A$ and $B$ are positively quadrant dependent, Lemma \ref{n11.1} implies that
\begin{align*}
\P(A\leq a, B\leq b)\geq \P(A\leq a)\P( B\leq b) 
=\P(A'\leq a)\P(B'\leq b)
=\P(A'\leq a,B'\leq b).
\end{align*}
It is easy to check that $\frac{\partial^2}{\partial x\partial y}h\geq 0$.
Hence, by Lemma \ref{lemma:Tchen1980}, for a fixed $\eps \in (0, 1-a)$,
\begin{align*}
 \E\left[e^{-f_S(Az)-zB}\right]&\geq 
  \E\left[e^{-f_S(A'z)-zB'}\right]
= \E\left[e^{-f_S(A'z)}\right]\E\left[e^{-zB'}\right]
=\E\left[e^{-f_S(Az)}\right]\E\left[e^{-zB}\right]\\
& \geq \E\left[e^{-f_S(Az)}\1_{\{A<a+\varepsilon\}}\right]e^{-f_B(z)}
\geq e^{-f_S((a+\varepsilon)z)-f_B(z)}\P(A<a+\varepsilon).
\end{align*}
This and \eqref{n11.2} imply that
\begin{equation*} 
f_{S}((a+\varepsilon)z)+f_B(z)-\log\P(A<a+\varepsilon)\geq f_S(z). 
\end{equation*}
Substituting $(a+\varepsilon)^k z$ for $z$ in the last formula yields
\begin{equation*} 
f_{S}((a+\varepsilon)^{k+1}z)+f_B((a+\varepsilon)^k z)-\log\P(A<a+\varepsilon)\geq f_S((a+\varepsilon)^k z). 
\end{equation*}
The telescoping sum argument gives
\begin{align}
\nonumber f_S(z)-f_S((a+\varepsilon)^{n+1}z)&\leq \sum_{k=0}^n f_B((a+\varepsilon)^kz)-(n+1)\log \P(A<a+\varepsilon)\\
&\leq \sum_{k=0}^\infty f_B((a+\varepsilon)^kz)-(n+1)\log \P(A<a+\varepsilon). \label{eq:twoSumsComparison}
\end{align}
Note  that 
$$\sum_{k=0}^{\infty } f_B((a+\varepsilon)^kz)=-\log \E\left[\exp\left(-z\sum_{k=0}^{\infty }(a+\varepsilon)^kB_k\right)\right].$$
It follows from Proposition \ref{propo:closureOnMultiplication} and
 Theorem \ref{prop:expectationOfLogarithm} that $\sum_{k=0}^{\infty }(a+\varepsilon)^kB_k$ is a finite $\IED^\rho_L(\Lambda_1)$-random variable, where
 $$ \Lambda_1=\lambda\left(\sum_{k=0}^{\infty}(a+\eps)^{k\rho/(1+\rho)}\right)^{1+\rho}
= \frac{\lambda} { (1- (a+\eps)^{\rho/(1+\rho)})^{1+\rho}}.
$$
Hence, by Theorem \ref{thm:Bruijin} we have 
\begin{equation}\label{eq:sumAssymptotics}
 \sum_{k=0}^{\infty } f_B((a+\varepsilon)^kz)\sim \left(1-(a+\eps)^{\rho/(1+\rho)}\right)^{-1}f_B(z).
\end{equation}


If we take $n=\left\lceil- \log z/\log (a+\varepsilon) \right\rceil$ then
$(a+\varepsilon)^2 \leq
(a+\varepsilon)^{n+1}z \leq 1$ and, therefore $|f_S((a+\varepsilon)^{n+1}z)|< c_1$. Also, 
\begin{align*}
|(n+1)\log \P(A<a+\varepsilon)| \leq
\frac{|2\log \P(A<a+\varepsilon)|}{\log (a+\varepsilon)}
\log z .
\end{align*}
These observations, the fact that $f_B$ is regularly varying at infinity with index $\rho/(1+\rho)>0$,
 \eqref{eq:twoSumsComparison} and \eqref{eq:sumAssymptotics}, imply that
\begin{align}\label{n11.5}
\limsup_{z\to \infty} \frac{f_S(z)}{f_B(z)}\leq \lim_{z\to\infty}\sum_{k=0}^\infty\frac{ f_B((a+\varepsilon)^kz)}{f_B(z)}=\left(1-(a+\eps)^{\rho/(1+\rho)}\right)^{-1}.
\end{align}
This completes the proof of \eqref{n11.3} because
the above estimate holds for all sufficiently small $\eps>0$ and we already have \eqref{n11.4}.

(b) Part (b) follows from Theorem \ref{thm:fixedPointRepresentation2} and part (a). 
\end{proof}

Methods similar to those in the proof of Theorem \ref{thm:seriesConvegenceABIndependent} were used in \cite{Kolodziejek} to analyze light-tailed solutions to $X\stackrel{d}{=}AX+B$.

We will now interpret the parameter $\Lambda$ in \eqref{n12.1} in a  way similar to that  in Proposition \ref{thm:lowerEnvelope}.

\thm{\label{thm:lowerEnvelope3}
Let $(A_i,B_i)_{i=1}^\infty$ be an i.i.d. sequence of two-dimensional  vectors with the following properties. 
\begin{enumerate}[(i)]
 \item $A_1$ and $B_1$ are nonnegative and positively quadrant dependent random variables. 
 \item There exists $\beta\in (0,1)$ such that $A_1\leq \beta$, a.s. 
\item $\E[(\log^+B_1)^s]<\infty$ for all $s>0$.
 \item $B_1$ is an \IGFT$^\rho_L(\lambda)$-random variable.
\end{enumerate}
Let $X_0=0$ and
\begin{align}\label{d28.1}
X_n=A_nX_{n-1}+B_n,\quad n\geq 1.
\end{align}
Recall $\Lambda$ defined in \eqref{n12.1} and let  $g$ be the asymptotic inverse of $x\mapsto x^{\rho}L(x)$ at $0$.
Then
$$
\liminf_{n\to\infty} \frac{X_n}{g(1/\log n)} =\Lambda^{1/\rho},\ a.s.
$$
}

\rm The structure of the proof will be similar to that  of Theorem \ref{thm:lowerEnvelope2}, however some new technical steps will be needed.
In all lemmas preceding the proof of Theorem \ref{thm:lowerEnvelope3}, we
will use the same notation and  make the same assumptions as in the theorem.

\lemma{\label{lemma:stochasticallyDecresingSequence}
\begin{enumerate}[(a)]
\item 
Let $S$ be defined as in \eqref{eq:solution:series2}. We have
$X_n\stackrel{d}{\to} S$ as $n\to\infty$ and 
\begin{equation}\label{eq:stochasticallyDecresingSequence}
\P(S\leq x)\leq \P(X_{n}\leq x)\leq \P(X_{n-1}\leq x),
\end{equation} 
for all $n\geq 1$ and $x\geq 0$.
\item Let $a=\essinf A_1$. For every $n\geq 1$, $X_n$ is an $\IGFT^\rho_L(\Lambda_n)$-random variable with  
\begin{align}\label{n12.5}
\Lambda_n=\lambda\left(\frac{1-a^{\frac{n\rho}{1+\rho}}}{1-a^{\rho/(1+\rho)}}\right)^{1+\rho}.
\end{align}
\end{enumerate}
 }
\begin{proof}

Set $S_n=\sum_{i=1}^n \left(\prod_{j=1}^{i-1}A_j\right)B_i$. Definition \eqref{d28.1} implies the following representation, $X_n=\sum_{i=1}^n \left(\prod_{j=n-i+2}^{n}A_j\right)B_{n-i+1}$.
This and the change of index $i\mapsto n-i+1$ easily show that
$X_n\stackrel{d}{=}S_n$. Therefore, $\P(X_{n}\leq x) =\P(S_n\leq x)$. Since $A_j$'s and $B_j$'s are non-negative, $S_n\uparrow S$ a.s. 
All claims made in part (a)  follow easily from these observations. 

The definition of the essential infimum $a$ and the assumption that $A_i$'s and $B_i$'s are non-negative imply that
\begin{align}\label{n12.3}
\P(S_n\leq x)\leq \P(a^{n-1}B_n+a^{n-2}B_{n-2}+\ldots + aB_2+ B_1\leq x).
\end{align}

We use the assumption that $(A_k,B_k)$, $k\geq 1$ are i.i.d., $(A_1,B_1)$ are positively quadrant dependent, and
Lemma \ref{n11.1} to see that for $\varepsilon>0$,
\begin{align}\label{n12.4}
\P(S_n\leq x) & \geq \P(S_n\leq x, A_1\leq a+\varepsilon, A_2\leq a+\varepsilon, \ldots , A_{n-1}\leq a+\varepsilon)\\
 & \geq \P\Big((a+\varepsilon)^{n-1}B_n+(a+\varepsilon)^{n-2}B_{n-2}+\ldots + (a+\varepsilon)B_2+ B_1\leq x, \notag \\
 & \qquad A_1\leq a+\varepsilon, A_2\leq a+\varepsilon, \ldots , A_{n-1}\leq a+\varepsilon\Big) \notag \\
  & \geq \P\Big((a+\varepsilon)^{n-1}B_n+(a+\varepsilon)^{n-2}B_{n-2}+\ldots + (a+\varepsilon)B_2+ B_1\leq x, \notag \\
 & \qquad  A_2\leq a+\varepsilon, \ldots , A_{n-1}\leq a+\varepsilon\Big) \P(A_1\leq a+\varepsilon) \notag \\
   & \geq \P\Big((a+\varepsilon)^{n-1}B_n+(a+\varepsilon)^{n-2}B_{n-2}+\ldots + (a+\varepsilon)B_2+ B_1\leq x, \notag \\
 & \qquad  A_3\leq a+\varepsilon, \ldots , A_{n-1}\leq a+\varepsilon\Big) \P(A_1\leq a+\varepsilon)\P(A_2\leq a+\varepsilon) \notag \\
 & \dots \notag \\
							&\geq \P((a+\varepsilon)^{n-1}B_n+(a+\varepsilon)^{n-2}B_{n-2}+\ldots + (a+\varepsilon)B_2+ B_1\leq x)\times \notag \\
							&\qquad\times \P(A_1\leq a+\varepsilon)^{n-1}. \notag
\end{align}

We have
\begin{align}\label{n12.2}
\sum_{k=0}^{n-1} (a^k)^{\rho/(1+\rho)}
=\frac{1-a^{n\rho/(1+\rho)}}{1-a^{\rho/(1+\rho)}},
\qquad 
\sum_{k=0}^{n-1} ((a+\eps)^k)^{\rho/(1+\rho)}
=\frac{1-(a+\eps)^{n\rho/(1+\rho)}}{1-(a+\eps)^{\rho/(1+\rho)}}.
\end{align}

Recall that $\P(X_n\leq x)=\P(S_n\leq x)$ and use \eqref{n12.3}, \eqref{n12.4}, \eqref{n12.2} and Proposition \ref{thm:sumOfVariablesTail} to see that
\begin{align*}
\limsup_{x\to 0^+} x^\rho L(x) \log \P(X_n\leq x) 
&=
\limsup_{x\to 0^+} x^\rho L(x) \log \P(S_n\leq x) \leq \lambda\left(\frac{1-a^{\frac{n\rho}{1+\rho}}}{1-a^{\frac{\rho}{1+\rho}}}\right)^{1+\rho},\\
\liminf_{x\to 0^+} x^\rho L(x)\log \P(X_n\leq x)
&=\liminf_{x\to 0^+} x^\rho L(x)\log \P(S_n\leq x)
\geq
\lambda\left(\frac{1-(a+\varepsilon)^{\frac{n\rho}{1+\rho}}}{1-(a+\varepsilon)^{\frac{\rho}{1+\rho}}}\right)^{1+\rho}.
\end{align*}
We complete the proof of (b)
by letting $\varepsilon\to 0$.
\end{proof}

\lemma{\label{lemma:easySideOfBorelCantelli2}

(i) For every $\varepsilon>0$,
$$\left\{X_n\leq g\left(\frac{\Lambda}{(1+\varepsilon)\log n}\right)\right\}$$
happens finitely often almost surely.

(ii)  We have
$$\liminf_{n\to\infty}\frac{X_n}{g(1/\log n)}  \geq \Lambda^{1/\rho}\ a.s.$$
}\begin{proof}
(i) For any $\eps>0$
there exists $\delta\in(0,1)$ such that $\gamma := (1-\delta)(1+\varepsilon/2)>1$. 
Assumption (ii) of Theorem \ref{thm:lowerEnvelope3} implies that $a<1$ so $\Lambda_n$ defined in \eqref{n12.5} converge to $\Lambda$ as $n\to\infty$.
By Lemma \ref{propo:characterization} and
part (b) of Lemma \ref{lemma:stochasticallyDecresingSequence},
there exist $C_\delta$, $n_0$ and $x_0>0$ such that $\P(X_{n_0}\leq x)\leq C_\delta e^{-\Lambda(1-\delta)/(x^\rho L(x))}$ for all $x\in (0,x_0)$.
By part (a) of Lemma \ref{lemma:stochasticallyDecresingSequence}, for $n\geq n_0$ and $x\in (0,x_0)$,
$$\P(X_n\leq x)\leq C_\delta e^{-\Lambda(1-\delta)/(x^\rho L(x))}.$$
It follows that, for large $n$,
\begin{align*}
 &\P\left(X_n\leq g\left(\frac{\Lambda}{(1+\varepsilon)\log n}\right)\right)
\leq C_\delta e^{-(1+\varepsilon/2)(1-\delta)\log n}=C_{\delta}n^{-\gamma}.
\end{align*}
Hence, 
$$\sum_{n=1}^{\infty}\P\left(X_n\leq g\left(\frac{\Lambda}{(1+\varepsilon)\log n}\right)\right)<\infty,$$
and the claim follows by the Borel-Cantelli lemma.

(ii) Part (i) implies that for every $\eps >0$,
$$\liminf_{n\to\infty}\frac{X_n}{g(1/\log n)}  \geq \frac{\Lambda^{1/\rho}}{(1+\varepsilon)^{1/\rho}}.$$
Part (ii) follows by letting $\eps\to0$. 
\end{proof}

\rm

From  definition \eqref{d28.1}, we obtain for $n\geq m$,
\begin{align}\label{d29.2}
 X_n&=B_n+A_nX_{n-1}=B_n+A_nB_{n-1}+ A_nA_{n-1}X_{n-2}=\ldots\\
&=\sum_{j=m+1}^n\left(\prod_{k=j+1}^n A_k\right)B_j + \left(\prod_{k=m+1}^nA_k\right)X_m.\notag
\end{align}
We rearrange terms and define new random variables,
\begin{equation}\label{eq:cutOffSequenceDefinition}
X_n^m:=X_n - \left(\prod_{k=m+1}^nA_k\right)X_m = \sum_{j=m+1}^n\left(\prod_{k=j+1}^n A_k\right)B_j.
\end{equation}
The notation $X^m_n$ is the same as in \eqref{d29.1} but the meaning is different. We have chosen the same notation for a different object because this will allow us to reuse a part of  the proof of Lemma \ref{lem:tailOfTheartificialSequence}

\lemma{\label{lem:representationOfTheArtificialSequence}\begin{enumerate}[(a)]
\item Let $Y_0=0$ and $Y_k:=X_{m+k}^m$ for $k\geq 1$. The sequence $(Y_k,k\geq 0)$ satisfies 
$$Y_{k+1}= A_{m+k+1}Y_k + B_{m+k+1}.$$
\item For every $m\geq 0$, the sequence
$(X_{m+k}^m,k\geq 0)$ has the same distribution as  $(X_{k},k\geq 0)$.
\end{enumerate}}
\begin{proof}
The claim in (a) follows from \eqref{eq:cutOffSequenceDefinition}. Part (b) follows from (a).
\end{proof}

\lemma{\label{lemma:artifficialSequence2}

Fix  $\varepsilon>0$ and suppose that $1<1+\delta< \sqrt{1+\eps}$ and $\alpha \in (1, 1+\delta)$. Events 
$$\left\{X_{\lfloor n^\alpha \rfloor}^{\lfloor (n-1)^\alpha \rfloor}\leq g\left(\frac{\Lambda(1+\varepsilon)}{\log n^{\alpha}}\right)\right\}$$
happen infinitely often a.s. }
\begin{proof}

We use Lemma \ref{lem:representationOfTheArtificialSequence} (b) 
and \eqref{eq:stochasticallyDecresingSequence} to see that
\begin{align*}
\sum_{n=1}^\infty \P\left(X_{\lfloor n^\alpha \rfloor}^{\lfloor (n-1)^\alpha \rfloor}\leq g\left(\frac{\Lambda(1+\varepsilon)}{\log n^{\alpha}}\right)\right)
=&\sum_{n=1}^\infty \P\left(X_{\lfloor n^\alpha \rfloor-\lfloor (n-1)^\alpha \rfloor}\leq g\left(\frac{\Lambda(1+\varepsilon)}{\log n^{\alpha}}\right)\right)\\
&\geq \sum_{n=1}^\infty \P\left(S\leq g\left(\frac{\Lambda(1+\varepsilon)}{\log n^{\alpha}}\right)\right).
\end{align*}
The random variable $S$ is \IGFT$^\rho_L(\Lambda)$ by Theorem \ref{thm:seriesConvegenceABIndependent}, so we can use
 Lemma \ref{propo:characterization} to write
$$
\P\left(S\leq g\left(\frac{\Lambda(1+\varepsilon)}{\log n^{\alpha}}\right)\right)
\geq\P\left(S\leq g\left(\frac{\Lambda(1+\delta)}{\log n^{\alpha/(1+\delta)}}\right)\right)\geq c_\delta n^{-\alpha/(1+\delta)} .
 $$
It follows that
 \begin{align*}
  &\sum_{n=2}^{\infty}\P\left(X_{\lfloor n^\alpha \rfloor}^{\lfloor (n-1)^\alpha \rfloor}\leq g\left(\frac{\Lambda(1+\varepsilon)}{\log n^{\alpha}}\right)\right) = \infty.
 \end{align*}
This, the fact that the random variables  $X_{\lfloor n^\alpha \rfloor}^{\lfloor (n-1)^\alpha \rfloor}$, $n\geq 2$, are jointly independent
and  the Borel-Cantelli lemma imply the claim made in the lemma.
\end{proof}

\rm

\lemma{\label{lem:tailOfTheartificialSequence2}
When $n\to\infty$, 
$$ \frac{X_{\lfloor n^\alpha \rfloor}-X_{\lfloor n^\alpha \rfloor}^{\lfloor (n-1)^\alpha \rfloor}}{g(\log \lfloor n^\alpha \rfloor)}\to 0,\quad \text {  a.s.  }$$
}

\begin{proof}

We first use \eqref{eq:cutOffSequenceDefinition} and then \eqref{d29.2} with $m=0$ to obtain
\begin{align*}
X_{\lfloor n^\alpha \rfloor}-X_{\lfloor n^\alpha \rfloor}^{\lfloor (n-1)^\alpha \rfloor}
= &\left(\prod_{k=\lfloor (n-1)^\alpha \rfloor+1}^{\lfloor n^\alpha \rfloor}A_k\right) X_{\lfloor (n-1)^\alpha \rfloor}\\
= &\left(\prod_{k=\lfloor (n-1)^\alpha \rfloor+1}^{\lfloor n^\alpha \rfloor}A_k\right) \left(\sum_{j=1}^{\lfloor (n-1)^\alpha \rfloor}\left(\prod_{k=j+1}^{\lfloor (n-1)^\alpha \rfloor} A_k\right)B_j\right)\\
=&\sum_{j=1}^{\lfloor (n-1)^\alpha \rfloor}\left(\prod_{k=j+1}^{\lfloor n^\alpha \rfloor} A_k\right)B_j.
\end{align*}
We apply assumption (ii) of Theorem \ref{thm:lowerEnvelope3},  a part of assumption (i) (namely, that $A_k$'s and $B_k$'s are nonnegative), and the change of index $i=\lfloor (n-1)^\alpha \rfloor-j$, to see that
\begin{align*}
X_{\lfloor n^\alpha \rfloor}-X_{\lfloor n^\alpha \rfloor}^{\lfloor (n-1)^\alpha \rfloor}
= &\sum_{j=1}^{\lfloor (n-1)^\alpha \rfloor}\left(\prod_{k=j+1}^{\lfloor n^\alpha \rfloor} A_k\right)B_j\leq \sum_{j=1}^{\lfloor (n-1)^\alpha \rfloor}\beta^{\lfloor n^\alpha \rfloor-j} B_j\\
=&\sum_{i=0}^{\lfloor (n-1)^\alpha \rfloor-1} \beta^{\lfloor n^\alpha \rfloor-\lfloor (n-1)^\alpha \rfloor+i} B_{\lfloor (n-1)^\alpha \rfloor-i}.
\end{align*}

The rest of the proof is the same as the proof of Lemma \ref{lem:tailOfTheartificialSequence} starting at $(\ref{eq:inequalityDiffReducedSequence})$, with $C=1$.
\end{proof}

\begin{proof}[Proof of Theorem \ref{thm:lowerEnvelope3}]
Note that $g$ is regularly varying with index $1/\rho$ at 0.
By Lemmas \ref{lemma:artifficialSequence2} and \ref{lem:tailOfTheartificialSequence2}, for every $\varepsilon>0$,
\begin{align*}
&\liminf_{n\to \infty}\frac{ X_n}{g(1/\log n)}
\leq \liminf_{n\to \infty} \frac{X_{\lfloor n^\alpha \rfloor}}{g(1/\log \left( \lfloor n^\alpha \rfloor\right))}\\
&\qquad=
\liminf_{n\to \infty}
\left( \frac{X_{\lfloor n^\alpha \rfloor}-X_{\lfloor n^\alpha \rfloor}^{\lfloor (n-1)^\alpha \rfloor}}{g(1/\log \left( \lfloor n^\alpha \rfloor\right))}
+  \frac{X_{\lfloor n^\alpha \rfloor}^{\lfloor (n-1)^\alpha \rfloor}}{g(1/\log  \left( \lfloor n^\alpha \rfloor\right))} \right)\\
&\qquad\leq \Lambda^{1/\rho}(1+\varepsilon)^{1/\rho}.
\end{align*}
Hence, $\liminf_{n\to \infty} X_n/g(1/\log (n) ) \leq \Lambda^{1/\rho}$ a.s. 
The theorem follows from this and Lemma \ref{lemma:easySideOfBorelCantelli2} (ii). 
\end{proof}

\rm

 The following theorem is a version of the well-known results by Kesten \cite{Kesten} and Goldie \cite{GoldieAAP}, 
formulated in \cite[Theorem 2.4.4]{X=AX+B}.

\thm{\label{thm:KestenGoldieResult}Assume that $(A,B)$ satisfy the following conditions.
\begin{enumerate}[(i)]
 \item $A\geq 0$, a.s., and the law of $\log A$ conditioned on $\{A>0\}$ is non-arithmetic, i.e., it is not supported on 
$a\Z$ for any $a> 0$.
\item There exists $\alpha>0$ such that $\E[A^{\alpha}]=1$, $\E[|B|^{\alpha}]<\infty$ and $\E[A^{\alpha}\log^+A]<\infty$.
\item $\P(Ax+B=x)<1$ for every $x\in \R$.
\end{enumerate}
Then the equation $X\stackrel{d}{=}AX+B$ has a solution. There exist constants $c_+,c_-$
such that $c_++c_->0$ and 
\begin{align}\label{n24.3}
\P(X>x)\sim c_+x^{-\alpha}\quad\textrm{and}\quad \P(X<-x)\sim c_-x^{-\alpha},\qquad \text {  when  }x\to\infty.
\end{align}
The constants $c_+,c_-$ are given by
$$c_+=\frac{1}{\alpha m_{\alpha}}\E\left[(AX+B)_+^{\alpha}-(AX)_+^{\alpha}\right],\quad c_-=\frac{1}{\alpha m_{\alpha}}\E\left[(AX+B)_-^{\alpha}-(AX)_-^{\alpha}\right],$$
where $m_{\alpha}=\E[A^{\alpha}\log A]$.
}\rm

\medskip
Next we will combine the results of Kesten and Goldie with our own.

\thm{\label{thm:KestenGoldieAndOurResults}
Assume that $(A,B)$ satisfy the following conditions.
\begin{enumerate}[(i)]
 \item $A$ and $B$ 
  are nonnegative, non-constant and positively quadrant dependent random variables. 
 \item There exists $\alpha>0$ such that $\E[ A^\alpha]=1$,  $\E[ A^\alpha\log A]<\infty$
and $\log A$ conditioned on $\{A>0\}$ is a non-arithmetic random variable.
 \item $B$ is an \IGFT$^\rho_L(\lambda)$-random variable
 and $\E[B^\alpha]<\infty$.
\end{enumerate}

Then the stochastic fixed point equation
$X\stackrel{d}{=}AX+B$
has a unique solution 
which is an $\IED^\rho_L(\Lambda)$ random variable, where
$\Lambda$ is defined in \eqref{n12.1}.

Moreover,
\begin{equation}\label{eq:tailInversePolynomial}
\lim_{x\to \infty}x^{\alpha} \P(X>x) 
\end{equation}
exists and is a positive number.}
\begin{proof}

We will show that assumptions of Theorem \ref{thm:seriesConvegenceABIndependent}
are satisfied. 
Since $x\mapsto \log x$ is concave and $A$ is non-constant, we have $\alpha\E[\log A]<\log \E[A^\alpha]=0$, so assumption (ii) of Theorem \ref{thm:seriesConvegenceABIndependent} holds. The other assumptions also hold so the first claim follows from Theorem \ref{thm:seriesConvegenceABIndependent}.

We note that assumptions (i) and (ii)  of Theorem \ref{thm:KestenGoldieResult} hold. We will verify assumption (iii).
The function
$\log A$ is non-arithmetic when conditioned on $\{A>0\}$,  hence $\P( A\ne 1)>0$. 
Since $\E[A^\alpha]=1$, there exists $\alpha \in (0,1)$ such that $\P(A\leq \alpha)>0$. Random variables $A$ and $B$ are positive quadrant dependent and $B$ is an $\IED^\rho_L(\lambda)$-random variable so
$$\P(Ax+B\neq x)\geq \P(Ax\leq \alpha x, B\leq(1-\alpha/2)x)\geq \P(A\leq \alpha)\P(B\leq(1-\alpha/2)x)>0.$$
All assumptions of 
Theorem \ref{thm:KestenGoldieResult} have been verified, so the claim 
$(\ref{eq:tailInversePolynomial})$ is a consequence of \eqref{n24.3}.
\end{proof}

\section{Dependent coefficients in the fixed point equation}\label{sec:example}

\rm

This section has a double purpose. First, we will explain how the questions studied in this paper arose in a different project. 
That project  is devoted to a rather different topic so we will only sketch some of its ideas.
Needless to say, we hope that our present results will be used to study other models.

Second, the fixed point equation \eqref{stochasticFixedPointEquation} coming from the other project has coefficients $A$ and $B$ dependent in a different way than in the previous sections of this paper. We plan to develop a theory for  such equations in a future article. Here we will limit ourselves  
to showing that the lack of positive quadrant dependence can make a substantial difference to the main results on \IGFT{} random variables.

\subsection{Motivation}

In the rest of this section, we will assume that the vector $(A,B)$
has the following density.
\begin{equation}
 \P(A\in da, B\in db)=\frac{e^{-\left(a^{1/2}-\frac{1}{2}\right)^2/b}-e^{-\left(a^{1/2}+\frac{1}{2}\right)^2/b}}{2\sqrt{\pi ab}}\cdot \frac{e^{-1 /(4b)}}{\sqrt{\pi}b^{3/2}},
\qquad a,b>0.\label{eq:exampleDensity}
\end{equation}

We will now explain how this density arose in a project on the Fleming-Viot type process (see \cite{BHM}). Our new results are in preparation but one can find the following basic scheme in \cite{extinctionOfFlemingViot}. Let: 
\begin{itemize}
 \item $W_1=(W_1(t):t\geq 0)$ and $W_2=(W_2(t):t\geq 0)$ be two independent Brownian motions;
 \item $\tau_1$ and $\tau_2$, respectively, be the first times $W_1$ and $W_2$ hit $0$;
 \item  $T=\min\{\tau_1,\tau_2\}$ and $Y=\max\{W_1(T),W_2(T)\}$.  
\end{itemize}
One can show that
$$\P(Y\in dy,T\in dt)=\frac{\exp\left(-\frac{(1-y)^2+1}{2t}\right)-
\exp\left(-\frac{(1+y)^2+1}{2t}\right)}{\pi t^2}dtdy, \qquad y,t>0.$$
The distribution given in $(\ref{eq:exampleDensity})$ is obtained by the substitution
$A=Y^{-2}$, $B=TY^{-2}$.\vspace{0.2cm}

Let $Y_n$ denote the position and let $T_n$ denote the time of the $n$-th renewal of the Fleming-Viot type process. 
Self-similarity of the process implies that $\left(T_n/Y_n^2\right)_{n\geq 1}$ is an iterated random sequence,
whose limiting behavior is described by the stochastic fixed point equation $X\stackrel{d}{=}AX+B$.  
We are interested in the right tail behavior of $Y_n/\sqrt{T_n}$ for large $n$, so we could show that 
$\liminf_{n\to \infty} Y_n/\sqrt{T_n \log \log T_n}$ is a constant. It turns out that this is equivalent to estimating $\P(X<x)$ 
as $x\to 0^+$.

\subsection{Dependent coefficients}
We start with some basic facts about the distribution defined in \eqref{eq:exampleDensity}.

Recall that $\sim$ means that the ratio of  two quantities converges to 1.

\propo{\label{prop:marginaNonindependentAB}
Assume that the vector $(A,B)$ has the distribution given by $(\ref{eq:exampleDensity})$.
\begin{enumerate}[(a)]
 \item The density of $A$ is
\begin{equation}
 \P(A\in da)=\frac{4}{\pi(4a^2+1)}\, da.\label{eq:densityOfA}
\end{equation}
\item The density of $B$ is
\begin{equation}
 \P(B\in db) =\frac1{\sqrt{\pi}}
\left(\int_{-\frac{1}{\sqrt{2b}}}^{\frac{1}{\sqrt{2b}}}\frac{e^{-v^2/2}}{\sqrt{2\pi }}dv\right)\, \frac{e^{-1/(4b)}}{b^{3/2}} db. \label{eq:densityOfB}
\end{equation}
Moreover,
\begin{equation}
 \P(B<x)\sim \int_0^x \frac{e^{-1/(4b)}}{\sqrt{\pi}b^{3/2}}db\quad \textrm{as}\ x\to 0^+,\label{eq:lowerTailOfB}
\end{equation}
\begin{equation}
 \P(B>x)\sim \frac{1}{\pi x}\quad \textrm{as}\ x\to \infty.\label{eq:upperTailOfB}
\end{equation}

Random variable $B$ is  \IGFT$^1_1(1/4)$.
\end{enumerate}
}
\begin{proof}
(a) We integrate the density of $(A,B)$ with respect to $b$ over $(0,\infty)$ to compute the density of $A$.
\begin{align*}
 \P(A\in da) &= \int_0^{\infty}
\frac{\exp\left(-\frac{\left(a^{1/2}-\frac{1}{2}\right)^2+\frac{1}{4}}{b}
\right)
-\exp\left(-\frac{\left(a^{1/2}+\frac{1}{2}\right)^2+\frac{1}{4}}{b}\right)}{2\pi \sqrt{a}b^2}  dbda\\
&=\frac{1}{2\pi \sqrt{a}}\left[\int_0^{\infty}
\frac{\exp\left(-\frac{\left(a^{1/2}-\frac{1}{2}\right)^2+\frac{1}{4}}{b}\right)}{b^2}  db-\int_0^{\infty}
\frac{\exp\left(-\frac{\left(a^{1/2}+\frac{1}{2}\right)^2+\frac{1}{4}}{b}
\right)}{b^2}  db\right]da\\
&=\frac{1}{2\pi \sqrt{a}}\left[\frac{1}{\left(a^{1/2}-\frac{1}{2}\right)^2+\frac{1}{4}}  -\frac{1}{\left(a^{1/2}+\frac{1}{2}\right)^2+\frac{1}{4}}\right]da\\
&=\frac{da}{\pi \left(a^2+\frac{1}{4}\right)}.
\end{align*}

(b) In the following calculation, we use the substitution $ u=\sqrt{a}$.
\begin{align*}
 & \P(B\in db)\\
&= \int_0^{\infty}\frac{\exp\left(-\frac{\left(a^{1/2}-\frac{1}{2}\right)^2}{b}\right)-
\exp\left(-\frac{\left(a^{1/2}+\frac{1}{2}\right)^2}{b}\right)}{2\sqrt{\pi ab}} da\, \frac{e^{-1/(4b)}}{\sqrt{\pi}b^{3/2}} db\\
&=\int_0^{\infty}\frac{\exp\left(-\frac{\left(u-\frac{1}{2}\right)^2}{b}\right)-\exp\left(-\frac{\left(u+\frac{1}{2}\right)^2}{b}\right)}{\sqrt{\pi b}} du\, \frac{e^{-1/(4b)}}{\sqrt{\pi}b^{3/2}} db\\
&=\left(\int_0^{\infty}\frac{\exp\left(-\frac{\left(u-\frac{1}{2}\right)^2}{b}\right)}{\sqrt{\pi b}}du-\int_0^{\infty}\frac{\exp\left(-\frac{\left(u+\frac{1}{2}\right)^2}{b}\right)}{\sqrt{\pi b}} du\right)\, \frac{e^{-1/(4b)}}{\sqrt{\pi}b^{3/2}} db\\
&=\left(\int_{-\frac{1}{2}}^{\infty}\frac{e^{-u^2/b}}{\sqrt{\pi b}}du-\int_{\frac{1}{2}}^{\infty}\frac{e^{-u^2/b}}{\sqrt{\pi b}} du\right)\, \frac{e^{-1/(4b)}}{\sqrt{\pi}b^{3/2}} db\\
&=\left(\int_{-\frac{1}{2}}^{\frac{1}{2}}\frac{e^{-u^2/b}}{\sqrt{\pi b}}du\right)\, \frac{e^{-1/(4b)}}{\sqrt{\pi}b^{3/2}} db\\
&=\left(\int_{-\frac{1}{\sqrt{2b}}}^{\frac{1}{\sqrt{2b}}}\frac{e^{-v^2/2}}{\sqrt{2\pi }}dv\right)\, \frac{e^{-1/(4b)}}{\sqrt{\pi}b^{3/2}} db.
\end{align*}
This proves \eqref{eq:densityOfB}.
We use \eqref{eq:densityOfB} and the following facts, 
$$\lim_{b\to\infty}
\sqrt{b/2}
\int_{-\frac{1}{\sqrt{2b}}}^{\frac{1}{\sqrt{2b}}}
\frac{e^{-v^2/2}}{\sqrt{2\pi }}dv=\frac{1}{\sqrt{2\pi }}\quad \textrm{and}\quad \lim_{b\to\infty}e^{-1/(4b)}= 1,$$
to conclude that, when $x\to \infty$,
$$\P(B>x)\sim \frac{1}{\pi}\int_x^{\infty}b^{-2}db=\frac{1}{\pi x}.$$
This proves \eqref{eq:upperTailOfB}.
Since
$$\lim_{b\to 0^+}\int_{-\frac{1}{\sqrt{2b}}}^{\frac{1}{\sqrt{2b}}}\frac{e^{-v^2/2}}{\sqrt{2\pi }}dv=1,$$
we obtain for $x\to 0^+$,
$$\P(B<x)\sim \int_0^x \frac{e^{-1/(4b)}}{\sqrt{\pi}b^{3/2}}db.$$
This proves \eqref{eq:lowerTailOfB}. The claim that $B$ is an $\IGFT^1_1(\frac{1}{4})$-random variable follows from \eqref{eq:lowerTailOfB} by the same arguments as in Lemma \ref{propo:inverseGamaDistribTail}.
\end{proof}

\lemma{\label{lemma:nonIndependentAandB}
If $(A_1,B_1)$ and $(A_2,B_2)$ are independent random vectors 
with the density \eqref{eq:exampleDensity} then
\begin{equation}
\limsup_{\varepsilon\to 0^+} \varepsilon\log \P(A_2B_1+B_2<\varepsilon) \leq -\frac{3}{10}<-\frac{1}{4}.\label{eq:nonIndependentTail} 
\end{equation}
}
\remark{
Suppose that
$(A_1,B_1)$ and $(A_2,B_2)$ are i.i.d., and 
 $A_1$ and $B_1$ are independent with distributions $(\ref{eq:densityOfA})$ and $(\ref{eq:densityOfB})$.
Then $\essinf (A_1)=0$ and $B_1$ is \IGFT$^1_1(\frac{1}{4})$.
By Corollary \ref{corol:RandomCoefficientIndependentTail}, $\lim_{\varepsilon\to 0^+} \varepsilon\log \P(A_2B_1+B_2<\varepsilon)=-\frac{1}{4}$.  
However, $(\ref{eq:nonIndependentTail})$ shows that we do not have the same conclusion when $A_1$ and $B_1$ are not independent. }
\begin{proof}[Proof of Lemma \ref{lemma:nonIndependentAandB}]
First, recall that by Lemma \ref{propo:characterization}, for every $\delta>0$ there exists $C>0$ such that for all $x>0$,
\begin{equation*}
 \P(B_1<x)\leq C\exp\left(-\left(\frac{1}{4}-\delta\right)x^{-1}\right). 
\end{equation*}
The second inequality in the following calculation is justified by the above formula. Later in the calculation, we will use the substitution $y=a^{1/2}$.
\begin{align*}
 & \P(A_2B_1+B_2<\varepsilon)\\
&= \int_0^{\varepsilon}\int_0^{\infty}\P(aB_1+b<\varepsilon)
\frac{\exp\left(-\frac{\left(a^{1/2}-\frac{1}{2}\right)^2+\frac{1}{4}}{b}\right)-\exp\left(-\frac{\left(a^{1/2}+\frac{1}{2}\right)^2+\frac{1}{4}}{b}\right)}{2\pi b^2\sqrt{ a}} \, da\, db\\
&\leq \int_0^{\varepsilon}\int_0^{\infty}\P\left(B_1<\frac{\varepsilon-b}{a}\right)\frac{\exp\left(-\frac{\left(a^{1/2}-\frac{1}{2}\right)^2+\frac{1}{4}}{b}\right)}{2\pi b^2\sqrt{ a}} \, da\, db\\
&\leq \int_0^{\varepsilon}\int_0^{\infty}
C\exp\left(-\left(\frac{1}{4}-\delta\right)\frac{a}{\varepsilon-b}\right)\frac{\exp\left(-\frac{\left(a^{1/2}-\frac{1}{2}\right)^2+\frac{1}{4}}{b}\right)}{2\pi b^2\sqrt{ a}} \, da\, db\\
&\leq \int_0^{\varepsilon}\int_0^{\infty}
C\exp\left(-\left(\frac{1}{4}-\delta\right)\frac{a}{\varepsilon}\right)\frac{\exp\left(-\frac{\left(a^{1/2}-\frac{1}{2}\right)^2}{\varepsilon}-\frac{1}{4b}\right)}{2\pi b^2\sqrt{ a}} \, da\, db\\
&= C\int_0^{\infty}\frac{1}{2\sqrt{\pi  a}}\exp\left(-\varepsilon^{-1}\left[\left(a^{1/2}-\frac{1}{2}\right)^2+\left(\frac{1}{4}-\delta\right)a\right]\right) \, da\int_0^{\varepsilon}\frac{e^{-1/(4b)}}{\sqrt{\pi} b^2}\, db\\
&= C\int_0^{\infty}\frac{1}{2\sqrt{\pi  a}}\exp\left(-\varepsilon^{-1}\left[\left(\frac{5}{4}-\delta\right)\left(a^{1/2}-\frac{2}{5-4\delta}\right)^2+\frac{1}{4}-\frac{1}{5-4\delta}\right]\right) \, da\ \times\\
&\qquad \times\int_0^{\varepsilon}\frac{e^{-1/(4b)}}{\sqrt{\pi} b^2}\, db\\
&\leq C\sqrt{2}\left[\frac{\varepsilon}{2}\left(\frac{5}{4}-\delta\right)^{-1}\right]^{1/2}\int_{-\infty}^{\infty}\frac{1}{\sqrt{2\pi \left[\frac{\varepsilon}{2}\left(\frac{5}{4}-\delta\right)^{-1}\right]}}
\exp\left(-\frac{\left(y-\frac{2}{5-4\delta}\right)^2}{2\left[\frac{\varepsilon}{2}\left(\frac{5}{4}-\delta\right)^{-1}\right]}\right) \, dy\ \times\\ 
& \qquad \times\exp\left(-\varepsilon^{-1}\left(\frac{1}{4}-\frac{1}{5-4\delta}\right)\right)\int_0^{\varepsilon}\frac{e^{-1/(4b)}}{\sqrt{\pi} b^2}\, db\\
&= C\sqrt{2}\left[\frac{\varepsilon}{2}\left(\frac{5}{4}-\delta\right)^{-1}\right]^{1/2}\exp\left(-\varepsilon^{-1}\left(\frac{1}{4}-\frac{1}{5-4\delta}\right)\right)\int_0^{\varepsilon}
\frac{e^{-1/(4b)}}{\sqrt{\pi} b^2}\, db.
\end{align*}
By the same argument as in the proof of Lemma \ref{propo:inverseGamaDistribTail},
\begin{align*}
\limsup_{\varepsilon\to 0^+} \varepsilon\log 
\int_0^{\varepsilon}
\frac{e^{-1/(4b)}}{\sqrt{\pi} b^2}\, db = -1/4.
\end{align*}
This and the previous estimate yield
$$\limsup_{\varepsilon\to 0^+} \varepsilon\log \P(A_2B_1+B_2<\varepsilon) \leq -\frac{1}{4}-\frac{1}{4}+\frac{1}{5-4\delta}.$$
The proof of the lemma is completed by letting $\delta \to 0^+$.
\end{proof}

\propo{
Assume that the vector $(A,B)$ has the distribution given by $(\ref{eq:exampleDensity})$.
The stochastic fixed point equation 
\begin{equation}
 X\stackrel{d}{=}AX+B, \label{eq:fixPointEquationNonIndependenExample}
\end{equation}
has a unique solution and we have 
\begin{equation}
\limsup_{\varepsilon\to 0^+} \varepsilon\log \P(X<\varepsilon) \leq -\frac{3}{10}<-\frac{1}{4}. \label{eq:fixepointSolutionUpeerEstimateLowerTail}
\end{equation}
}
\begin{proof}
We use the substitution $a=x/2$ in the following calculation.
\begin{align*}
 \E[\log A]&=\int_0^{\infty} \frac{4\log a}{\pi (4a^2+1)}da 
 =2\int_0^{\infty} \frac{\log (x/2)}{\pi (x^2+1)}dx =2\int_0^{\infty} \frac{\log x-\log 2}{\pi (x^2+1)}dx\\
&=-2\log 2\int_0^{\infty} \frac{1}{\pi (x^2+1)}dx+\int_0^{\infty} \frac{\log x}{\pi (x^2+1)}dx.
\end{align*}
It is easy to see that $\int_0^{\infty} \frac{1}{\pi (x^2+1)}dx =1/2$. The substitution $y=x^{-1}$
yields $\int_0^{\infty} \frac{\log x}{\pi (x^2+1)}dx = -\int_0^{\infty} \frac{\log y}{\pi (y^2+1)}dy$, so these integrals must be equal to $0$. Hence, $\E[\log A]=-\log 2<0$. 

It follows from $(\ref{eq:upperTailOfB})$ that $\E[\log^+B]<\infty$.
The assumptions of Theorem \ref{thm:fixedPointRepresentation2} are satisfied so \eqref{eq:fixPointEquationNonIndependenExample} has  a unique solution.

Suppose that $X$ is the  solution to $(\ref{eq:fixPointEquationNonIndependenExample})$.
It is non-negative because it has the representation \eqref{eq:solution:series2}, where all random variables are non-negative.
Suppose that $(A_1,B_1)\stackrel{d}{=}(A,B)$, $(A_2,B_2)\stackrel{d}{=}(A,B)$, and $(A_1,B_1)$, $(A_2,B_2)$ and $X$ are jointly independent. Then
$$\P(X<\varepsilon)=\P(A_1X+B_1<\varepsilon)
=\P(A_2(A_1X+B_1)+B_2<\varepsilon)
\leq \P(A_2B_1+B_2<\varepsilon).$$
Now $(\ref{eq:fixepointSolutionUpeerEstimateLowerTail})$ follows from Lemma \ref{lemma:nonIndependentAandB}. 
\end{proof}

\remark{One can actually show that 
\begin{equation*}
 \lim_{\varepsilon\to 0^+} \varepsilon\log \P(X<\varepsilon)=-\frac{1}{2}.
\end{equation*}
However, the proof would take several additional pages so we will only sketch it. 
An appropriate modification of the argument in the proof of Lemma \ref{lemma:nonIndependentAandB}  shows that for the sequence $X_n=A_nX_{n-1}+B_n$ with $X_0=0$, we have 
$\limsup_{\varepsilon\to 0^+} \varepsilon\log \P(X_n<\varepsilon) \leq \tau_n$, where $\tau_n\downarrow -\frac{1}{2}$. 
Using the fact that $\P(X<\varepsilon)\leq \P(X_n<\varepsilon)$ it follows that
$$\limsup_{\varepsilon\to 0^+} \varepsilon\log \P(X<\varepsilon) \leq-\frac{1}{2}.$$
On the other hand for every $\delta>0$, one can find a bounded positive function $g_{\delta}$ on $(0,\infty)$
such that 
$$\P(Ax+B<\varepsilon)\geq g_{\delta}(x)e^{-\left(\frac{1}{2}+\delta\right)\varepsilon^{-1}},$$
for all $x> 0$. Hence, $\P(AX+B<\varepsilon)\geq \E[g_{\delta}(X)]e^{-\left(\frac{1}{2}+\delta\right)\varepsilon^{-1}}$, and, therefore,
$$\liminf_{\varepsilon\to 0^+} \varepsilon\log \P(X<\varepsilon) \geq-\left(\frac{1}{2}+\delta\right).$$
The claim  follows by letting $\delta\to 0^+$. }

\section*{Acknowledgements}\rm
The authors would like to thank Bojan Basrak, Hrvoje Planini\'c and anonymous referees for the most useful advice.
The third author would like to thank the Department of Mathematics at the University of Washington in Seattle, where the project took place, for the hospitality.
The third author is also grateful to the Microsoft Corporation for allowance on Azure cloud service where he ran many simulations.

Research of the first author was supported in part by Simons Foundation Grant 506732. 
 Research of the third author was supported in part by Croatian Science Foundation grant 3526.

\bibliographystyle{plain}
\bibliography{fixed}

\begin{thebibliography}{10}

\bibitem{extinctionOfFlemingViot}
Mariusz Bieniek, Krzysztof Burdzy, and Soumik Pal.
\newblock Extinction of {F}leming-{V}iot-type particle systems with strong
  drift.
\newblock {\em Electron. J. Probab.}, 17:no. 11, 15, 2012.

\bibitem{regularVariation}
N.~H. Bingham, C.~M. Goldie, and J.~L. Teugels.
\newblock {\em Regular variation}, volume~27 of {\em Encyclopedia of
  Mathematics and its Applications}.
\newblock Cambridge University Press, Cambridge, 1989.

\bibitem{BrockwellDavis}
Peter~J. Brockwell and Richard~A. Davis.
\newblock {\em Time series: theory and methods}.
\newblock Springer Series in Statistics. Springer-Verlag, New York, second
  edition, 1991.

\bibitem{X=AX+B}
Dariusz Buraczewski, Ewa Damek, and Thomas Mikosch.
\newblock {\em Stochastic models with power-law tails. The equation $X=AX+B$}.
\newblock Springer Series in Operations Research and Financial Engineering.
  Springer, 2016.

\bibitem{BHM}
Krzysztof Burdzy, Robert Ho\l{}yst, and Peter March.
\newblock A {F}leming-{V}iot particle representation of the {D}irichlet
  {L}aplacian.
\newblock {\em Comm. Math. Phys.}, 214(3):679--703, 2000.

\bibitem{diaconisFreedman}
Persi Diaconis and David Freedman.
\newblock Iterated random functions.
\newblock {\em SIAM Rev.}, 41(1):45--76, 1999.

\bibitem{GoldieAAP}
Charles~M. Goldie.
\newblock Implicit renewal theory and tails of solutions of random equations.
\newblock {\em Ann. Appl. Probab.}, 1(1):126--166, 1991.

\bibitem{HoffBayes}
Peter~D. Hoff.
\newblock {\em A first course in {B}ayesian statistical methods}.
\newblock Springer Texts in Statistics. Springer, New York, 2009.

\bibitem{Kesten}
Harry Kesten.
\newblock Random difference equations and renewal theory for products of random
  matrices.
\newblock {\em Acta Math.}, 131:207--248, 1973.

\bibitem{Kolodziejek}
Bartosz Ko{\l}odziejek.
\newblock On perpetuities with light tails.
\newblock {\em Adv. in Appl. Probab.}, 50(4):1119--1154, 2018.

\bibitem{tchen1980}
Andr\'e~H. Tchen.
\newblock Inequalities for distributions with given marginals.
\newblock {\em Ann. Probab.}, 8(4):814--827, 1980.

\end{thebibliography}

\end{document}